\documentclass[11pt]{amsart}
\usepackage[applemac]{inputenc}
\usepackage[T1]{fontenc}
\usepackage{ulem}
\usepackage{amsmath,esint,color}
\usepackage{amsthm}
\usepackage{amsfonts}
\usepackage{amssymb}
\usepackage{bigints}
\usepackage{fullpage}
\usepackage{enumerate}

\usepackage[colorlinks=true, pdfstartview=FitV, linkcolor=blue, citecolor=blue, urlcolor=blue]{hyperref}

\date{}
\definecolor{sah}{rgb}{0.66,0.33, 0.04}
\definecolor{adel4}{cmyk}{1,0,0,0}
\definecolor{adel3}{rgb}{0.66,0.33, 0.04}
\definecolor{adel1}{cmyk}{0,0.20,1,0}
\definecolor{adel2}{cmyk}{0,0.40,1,0.30}
\definecolor{adel0}{rgb}{0.99,0.60, 0.30}
\definecolor{trut}{rgb}{0.99,0.80, 0.00}
\definecolor{trus}{rgb}{0.00, 0.50, 0.00}
 \definecolor{trust}{rgb}{0.99, 0.99, 0.80}
\definecolor{MaCouleur}{rgb}{0,0.9,0.3}

\newcommand{\CC}{\mathbb{C}}  
\newcommand{\NN}{\mathbb{N}}  
\newcommand{\RR}{\mathbb{R}}  
\newcommand{\EE}{\varepsilon}

\newtheorem{definition}{Definition}
\newtheorem*{theorem}{Main Theorem}
\newtheorem{proposition}{Proposition}
\newtheorem{lemma}{Lemma}
\newtheorem{remark}{Remark}
\newtheorem{coro}{Corollary}

   \title[]{Existence of corotating and counter-rotating vortex pairs for active scalar equations}

\author[T. Hmidi]{Taoufik Hmidi}
\address{IRMAR, Universit\'e de Rennes 1\\ Campus de
Beaulieu\\ 35~042 Rennes cedex\\ France}
\email{thmidi@univ-rennes1.fr}
\author[J. Mateu]{Joan Mateu}
\address{Departament de Matem\`{a}tiques\\
Universitat Aut\`{o}noma de Barcelona\\
08193 Bellaterra, Barcelona, Catalonia} \email{ mateu@mat.uab.cat}

\begin{document}

\begin{abstract}
In this paper, we study    the existence of  corotating and counter-rotating pairs of simply connected patches  for Euler equations and the $(\hbox{SQG})_\alpha$ equations with $\alpha\in (0,1).$ From the numerical experiments implemented  for Euler equations in \cite{DZ, humbert, S-Z} it is conjectured  the existence of  a curve of steady vortex pairs  passing through the point vortex pairs. There are some analytical proofs based on variational principle \cite{keady, Tur}, however   they do not give enough information about the pairs such as the uniqueness or the topological structure of each single vortex. We  intend in this paper to give direct proofs confirming the numerical experiments and extend these results for the $(\hbox{SQG})_\alpha$ equation when $\alpha\in (0,1)$. The proofs rely on  the contour dynamics equations combined with a   desingularization of the point vortex pairs and the application of the  implicit function theorem.  \end{abstract}

\maketitle{}
\tableofcontents
\section{Introduction}
The present work deals with the dynamics of vortex pairs  for some nonlinear transport  equations arising in fluid dynamics. The equations that we shall consider are the   generalized   surface quasi-geostrophic  equations  which describe the evolution  of the potential temperature $\theta$ through the system,
\begin{equation}\label{sqgch2}
\left\{ \begin{array}{ll}
\partial_{t}\theta+v\cdot\nabla\theta=0,\quad(t,x)\in\RR_+\times\RR^2, \\
v=-\nabla^\perp(-\Delta)^{-1+\frac{\alpha}{2}}\theta,\\
\theta_{|t=0}=\theta_0.
\end{array} \right.
\end{equation}
Here $v$ refers to the velocity field, $\nabla^\perp=(-\partial_2,\partial_1)$  and $\alpha$ is a real parameter taken in  $[0,2)$. 
The  operator   $(-\Delta)^{-1+\frac{\alpha}{2}}$ is  of convolution type and defined as follows
$$
(-\Delta)^{-1+\frac{\alpha}{2}}\theta(x)=\int_{\mathbb{R}^2}K_\alpha(x-y) \theta(y)dy
$$
with 
 \begin{equation}
K_\alpha(x)=\left\{ \begin{array}{ll}
-\frac{1}{2\pi}\log|x|,\quad \hbox{if}\quad \alpha=0\\
\frac{C_\alpha}{2\pi}\frac{1}{|x|^\alpha}, \quad \hbox{if}\quad \alpha\in (0,2)
\end{array} \right.
\end{equation} 
and  $C_\alpha=\frac{\Gamma(\alpha/2)}{2^{1-\alpha}\Gamma(\frac{2-\alpha}{2})}$ where $\Gamma$ stands for the gamma function.  Note that this   model  was proposed by C\'ordoba { et al.} in \cite{C-F-M-R} as an interpolation between Euler equations and the surface quasi-geostrophic model (SQG)  corresponding to $\alpha=0$ and $\alpha=1$, respectively. We mention  that the SQG equation is used in   \cite{Held,Juk} to describe the atmosphere  circulation  near the tropopause and  to track the  ocean dynamics in  the upper layers  \cite{Lap}.  The   mathematical  analogy with the classical  three-dimensional incompressible Euler equations was pointed out in   \cite{C-M-T}.

In the last few years there has been a growing interest in the mathematical study of these active scalar 
equations. 
Local well-posedness of classical solutions  has been discussed  in various function spaces. For instance, 
this was implemented  in the framework of Sobolev spaces  \cite{C-C-C-G-W}, however,
the global existence is still an open problem except for Euler equations. The second  restriction with the  $(\hbox{SQG})_\alpha$ equation concerns the construction of Yudovich solutions-- known to exist globally in time for  Euler equations \cite{Y}-- which remains unsolved even locally in time.  The main difficulty  is due to the velocity which  is in general singular and scales  below the Lipschitz class. Nonetheless, one can say more about this issue for some special class of concentrated vortices. More precisely,  when the initial data is a single  vortex patch, that is, $\theta_0(x)=\chi_D$ is  the characteristic function of a bounded simply connected smooth domain $D$,  there is a unique local solution in the patch form $
 \theta(t)=\chi_{D_t}. $ In this case,  the boundary   motion of the domain $D_t$  
 is described by the contour dynamics formulation. Indeed, 
 the Lagrangian parametrization $\gamma_t:\mathbb{T}\to \partial D_t$ obeys to the 
 following integro-differential equations
 $$
 \partial_t\gamma_t(w)=\int_{\mathbb{T}}K_\alpha\big(\gamma_t(w)-\gamma_t(\xi)\big)\gamma_t^\prime(\xi)d\xi.
 $$
For more details, see
\cite{C,Gan,Ro}.  The global persistence of the boundary regularity is established for
Euler equations by Chemin \cite{C}, we refer also to  the paper of Bertozzi and Constantin \cite{B-C} 
for another proof. However for $\alpha>0$ only local persistence result is known  and  numerical 
experiments carried out in \cite{C-F-M-R} reveals  a   singularity formation  in finite  time.
Let us mention that  the contour dynamics equation remains  locally well-posed when the domain of
the initial data is composed of     multiple patches with different magnitudes in each component.
  
In this paper  we shall focus on steady single and multiple patches moving without changing shape, 
called relative equilibria or  V-states according to the terminology of Deem and Zabusky. Their  dynamics is seemingly
simple flow configurations described  by rotating or translating motion but it is  immensely rich and 
exhibits complex behaviors. There is abundant literature dealing with 
 numerical and analytical
structures for the 
isolated rotating patches and the
first example  goes back to  \mbox{Kirchhoff \cite{Kirc}} who proved  for Euler equations  that an ellipse 
of semi-axes $a$ and $b$  rotates uniformly  with the  angular velocity $\Omega = ab/(a+b)^2$. About one 
century later, Deem and Zabusky \cite{DZ}  provided strong numerical evidence  for the existence of 
rotating patches  with $m$-fold symmetry for the  integers $m \in\big\{3,4,5\big\}$. Recall that a 
domain is said  $m$-fold symmetric if it is invariant by the action of the dihedral \mbox{group $\textnormal{D}_m$.}
Few years later, Burbea gave in \cite{Bur} an  analytical  proof and showed  for any integer
$m\geq2$ the existence of a curve of V-states with $m$-fold symmetry bifurcating from Rankine vortex
at the angular velocity $\frac{m-1}{2m}$. The proof relies on  the use of complex analysis tools 
combined with the bifurcation theory. The regularity of the boundary   close to  Rankine vortices 
has been discussed very recently by the authors and Verdera in \cite{HMV} and where it was proved that 
the boundary is $C^\infty$ and convex. It seems that the boundary is actually  analytic according to 
the recent  result of Castro, C\'ordoba and G\'omez-Serrano \cite{Cor1}.  We also refer  to the paper
\cite{WOZ} where it is proved that corners with  right angles  is the only plausible scenario for  the 
limiting V-states. It is worthy to  point out that Burbea's  approach has been successfully implemented for 
the $(\hbox{SQG})_\alpha$ equations in \cite{Cor1,H-H} but with much more delicate computations. Similarly to the case $\alpha=0,$
we find countable family of bifurcating curves at some known angular velocities related to
gamma function.  In the same context, it turns out that for Euler equations a second bifurcation of countable branches  from the ellipses occurs but  the shapes have  in fact  less symmetry and being at most two-folds.  This was first observed numerically 
in \cite{Kam, Luz} and analytical proofs were recently discussed in \cite{Cor,HM0}.  
Another valuable  investigation has been devoted to the existence of doubly connected V-states where the rotating patches have  only one hole. In this case the boundary is comprised  of two Jordan curves obeying to   two coupled singular  nonlinear equations and thereby   the dynamics acquires more richness  and significant behaviors. The existence of such structures  was first accomplished 
for Euler equations in \cite{H-F-M-V} by using bifurcation tools in the spirit of Burbea's approach. Roughly speaking, for higher symmetry $m$  we get 
 two branches of $m-$fold V-states bifurcating from the annulus $\big\{b<|z|<1\big\}$ and numerical experiments 
about the limiting  V-states reveal different plausible  configurations depending on the size of the parameter $b$. Later, this result has been  
extended for the $(\hbox{SQG})_\alpha$ equations in \cite{DHH} for $\alpha\in (0,1)$ which surprisingly
exhibit various completely new behaviors compared to Euler equations. For example we  find rotating 
patches with negative and positive angular velocities for any $\alpha\in(0,1)$. It is worthy to mention 
that the bifurcation in the preceding cases is obtained under the transversality assumption 
of Crandall-Rabinowitz corresponding to simple nonlinear eigenvalues.
However the bifurcation in the degenerate case where there is  crossing eigenvalues 
is more complicate and  has been  recently solved\mbox{ in \cite{HM}. }

The main task  of this paper
is to deal with non connected V-states where the bifurcation arguments discussed above are out of use.  To be more precise, we shall be concerned with vortex pairs  moving without deformation.  This is   
 a fundamental and rich  subject in vortex dynamics 
and they serve for instance to model trailing vortices  behind the  wings of aircraft in steady horizontal flight  or to describe  the interaction between isolated vortex and a solid wall. We point out that the literature is very abundant and it  is by no means an easy  task to collect and recall all the results done in this field. Therefore we shall restrict the discussion  to  the cases of  counter-rotating and corotating vortices and recall some results that fit with our main goal. In the first case, the most common studied  configuration is  two  symmetric vortex  pairs with opposite circulations moving steadily with constant speed in a fixed direction. Notice that an explicit  example is given by a pair of point vortices with opposite circulations  which  translates steadily with the speed $U_{sing}=\frac{\gamma}{2\pi d},$ where $d$ is the distance separating the point vortices and $\gamma$ is the magnitude, see for instance \cite{Lamb}. Another nontrivial explicit example  of touching counter-rotating vortex pair was discovered by Lamb \cite{Lamb}, where the vortex is not uniformly distributed but has  a  smooth compactly supported profile related to Bessel functions of the first kind. Later,  Deem and Zabusky \cite{DZ} and Pierrehumbert \cite{humbert} provided numerically a class of translating vortex pairs of symmetric patches and they conjectured  the existence of a curve of translating symmetric pair of simply connected patches emerging  from   two point vortices and ending with two touching patches at right angle.  We mention  that  \mbox{Keady \cite{keady}} used a variational principle  in order to explore  the existence part and give asymptotic estimates for some significant functionals such as the excess kinetic energy and the speed of the pairs.  The basic idea is to  maximize the excess kinetic energy supplemented with some additional constraints and to show the existence of a maximizer taking the form of a pair of vortex patches in the spirit of the paper of  \mbox{Turkington \cite{Tur}. }However,  this approach does not give sufficient information on the structure of the pairs. For example the uniqueness of the maximizer  is left open and the topology of the patches is not well-explored, and it is not clear from the proof whether or not  each single patch is simply connected as it is suggested numerically.   
Concerning the corotating vortex pair, which consists of two symmetric patches with the same 
circulations  and rotating about the centroid of the system with constant angular velocity, it was 
investigated numerically by Saffman and Szeto in  \cite{S-Z}. They showed that when far apart, 
the vortices are almost circular and when the distance between them decreases they become more 
deformed until they touch. We remark that  a pair of point vortices far away at  a distance $d$ and  
with the same  magnitude $\gamma$  rotates   steadily with the angular velocity $\Omega_{sing}=\frac{\gamma}{\pi d^2}.$  By using variational principle, Turkington gave in \cite{Tur}  an analytic proof of the existence of corotating vortex pairs but    this general approach does not give enough precision   on the topological structure of each vortex patch similarly to the translating case commented  before. 
Note  that in the same direction Dritschel \cite{Drit}  calculated  numerically  V-states of vortex pairs with different 
shapes and discussed their linear stability. Very recently,  Denisov established in \cite{Denisov} 
for a modified Euler equations the existence of corotating simply connected vortex patches and analyzed 
the contact point of  the limiting V-states.  To end this short discussion, we want to  emphasize once again that the subject of vortex pairs has been intensively 
studied during the past and it is difficult  to track, know  and recall here  everything  written about. 
So, we have only selected some basic results  and the reader can find more details not only 
in this subject but also in some 
other connected topics in \cite{Ar, Burton, Burton1, gallay, Gallay1,Luz, Marc,New, Nor,  Saf, Smets} 
and the references therein.

In the current paper we intend to give  direct proofs for the existence of corotating and counter-rotating vortex pairs using the contour dynamics equations. We shall also extend these results to  the $(\hbox{SQG})_\alpha$ equations  for $\alpha\in (0,1)$. Now we shall  fix some notations before stating our main result. 
Let $0<\EE<1,$ $d>2$ and take a small simply connected domain $D_1$ containing the origin and contained in the open ball $D(0,2)$ centered at the origin and with radius $2$. Define  \begin{equation}\label{initi1}
\theta_{0,\EE}=\frac{1}{\EE^2}\chi_{D_1^\EE}+\delta \frac{1}{\EE^2}\chi_{D_2^\EE},\quad D_1^\EE=\EE D_1,\quad D_2^\EE=-D_1^\EE+2d,
\end{equation}
where the number $\delta$ is taken in $\{\pm1\}$. As we can readily observe, this initial data is composed of symmetric pair of simply connected patches with equal or opposite circulations. \\ The main result of the paper  is the following.
\begin{theorem}
Let $\alpha\in[0,1),$ there exists $\varepsilon_0>0$ such that the following results hold true.
\begin{enumerate}
\item Case $\delta=1$. For any $\EE\in(0,\EE_0]$ there exists a  strictly convex domain $D_1^\EE$  at least of  class $C^1$ such that $\theta_{0,\EE}$ in \eqref{initi1} generates a  corotating vortex pair for 
\eqref{sqgch2}.
\item Case $\delta=-1$. For any $\EE\in(0,\EE_0]$ there exists a strictly  convex domain $D_1^\EE$ of  class $C^1$ such that $\theta_{0,\EE}$ generates a counter-rotating vortex  pair for \eqref{sqgch2}.
\end{enumerate}
\end{theorem}
Before giving the basic ideas of the proofs some remarks are in order.
\begin{remark}
The domain $D_1^\EE$ is a small perturbation of the disc $D(0,\EE)$ , centered at zero and of radius $\EE$. Moreover, it can be described by the conformal parametrization $\phi_\EE:\mathbb{T}\to \partial D_1^\EE$ which belongs for $0<\alpha<1$ to  $C^{2-\alpha}(\mathbb{T})$  and for $\alpha=0$  to $C^{1+\beta}$ for any $\beta\in (0,1),$ and  satisfies
$$
\phi_\EE(w)=\EE w+\EE^{2+\alpha}f_\EE(w)\quad\hbox{with}\quad \|f_\EE\|_{C^{2-\alpha}}\le 1.
$$
Therefore  the boundary of each V-state  is at least $C^1$.  Note that with slight modifications we can adapt the proofs and show that the domain $D_1^\EE$ belongs to $C^{n+\beta}$ for any fixed $n\in\NN$. Of course, the size of $\EE_0$ depends on the parameter $n$ and cannot be uniform; it shrinks to zero as $n$ grows to infinity. 
 However, we expect the boundary to be analytic meaning that the conformal mapping possesses a holomorphic extension in $D(0,r)^c$ for some $0<r<1.$ The ideas developed in the recent paper \cite{Cor} might be useful to confirm such expectation. 
\end{remark}
\begin{remark}
 In the setting  of the vortex patches   the global existence with smooth boundaries is not known for $\alpha\in(0,2)$. In \cite{Cor1,DHH,H-H} we exhibit  the first nontrivial examples of simply connected and doubly connected V-states  which are  periodic in time. We find here another  class of global solutions which are the vortex pairs.
 \end{remark}

\begin{remark}
The proof is valid for $\alpha\in[0,1)$ but we expect  that the result remains true for $\alpha\in [1,2)$. We believe that the use of the spaces introduced in \cite{Cor1} could  be helpful for solving these cases.
\end{remark}

\begin{remark}
As we shall see later, we can unify the formalism leading to the existence of corotating patches with the 
point vortex model. The latter one is obtained when $\EE=0$ in which case we find the classical result
which says  that two point vortices at distance $2d$ and with the  same magnitude rotate uniformly about their center with the angular velocity
$ \Omega_{sing}^\alpha=\frac{\alpha C_\alpha}{\pi(2d)^{2+\alpha}}\cdot$ However when they have opposite signs they  exhibit a uniform translating motion  with the speed  $ U_{sing}^\alpha=\frac{\alpha C_\alpha }{\pi(2d)^{1+\alpha}}\cdot$
\end{remark}
Next we shall sketch the basic ideas used to prove the main result. 
We will just restrict the discussion  to the corotating pairs for  Euler equations   
since the proofs for the remaining cases  follow the same lines but with much more involved computations.
The proof relies on the desingularization of point vortex pairs combined with the implicit function theorem.
We  first formulate the equations governing the  corotating vortex pairs using  the complex variable, and 
we shall see later  in \mbox{Section \ref{st-vor}}  more details. However, we think that at this stage  it is convenient to sketch the arguments needed to get the contour dynamics equations of the boundary. Let $D_1$ and $D_2$ be  two disjoint simply connected domains and $a,b$ be two non vanishing real numbers. Then the initial datum
$$
\theta_0=a \chi_{D_1} +b\chi_{D_2}
$$
gives rise to a rotating vortex pair about the point $(d,0)$ and with the angular velocity $\Omega$ if the solution $\theta(t)$ of \eqref{sqgch2}  takes the form
$$
\theta(t,z-d)=\theta_0\big(e^{it \Omega}(z-d)\big).
$$
Therefore, inserting this expression  into the  equation \eqref{sqgch2} we find
$$
\Big(v_0(z)-\Omega(z-d)\Big)\cdot\nabla \theta_0(z)=0
$$
with $v_0$  the velocity associated to $\theta_0.$  Using the patch structure, the preceding equation reduces to
$$
\Big(v_0(z)-\Omega(z-d)\Big)\cdot \vec{n}(z)=0, \quad \forall z\in \partial D_1\cup\partial D_2
$$
where $\vec{n}(z)$ is a normal vector to the boundary.
Combining Biot-Savart law with Green-Stokes formula we find
$$
\overline{v_0(z)}=
 \frac{a}{4\pi}\int_{ \partial D_1} \frac{\overline{\xi}-
\overline{z}}{\xi-z}\, d\xi+\frac{b}{4\pi}\int_{ \partial D_2} \frac{\overline{\xi}-
\overline{z}}{\xi-z}\, d\xi, \quad \forall z \in \mathbb{C}.
$$
Consequently, using the special structure \eqref{initi1} combined with some   elementary transformations in the complex plane we deduce that  the two equations governing the boundary are actually equivalent to the following equation 
$$
 \textnormal{Re}\Big\{ \Big(2\Omega\big(\EE\overline{z}-d)+I_\EE(z)\Big)\, \vec{\tau}(z)\Big\}=0,\quad\forall  z\in \partial D_1
$$
with $\vec{\tau}(z)$ being the unit tangent vector to the boundary $ \partial D_1$
positively oriented and 
$$
I_\EE(z)=\frac{1}{2i\pi \varepsilon}\int_{\partial D_1}\frac{\overline{\xi}-\overline{z}}{\xi-z}d\xi-\frac{1}{2i\pi}\int_{\partial D_1^\varepsilon}\frac{\overline{\xi}}{\EE\xi+\EE z-2d}d\xi.
$$
The basic idea of the proof is to extend the functional defining the vortex pairs beyond $\EE=0$ corresponding to point vortex pairs and afterwards  to apply the implicit function theorem. As we can see, the first integral term  in $I_\EE(z)$ is singular and to remove the singularity we should seek for domains which are slight perturbation of the unit disc with a small amplitude of order $\EE$. In other words, we   look for a conformal  parametrization of  $D_1$ in the form
$$
\forall \, w\in \mathbb{T},\quad \phi_\EE(w)=w+\EE f(w)
$$ 
where the Fourier expansion of $f$ takes the form
$$
f(w)=\sum_{n\geq1} a_n w^{-n}, \quad a_n\in \RR\quad \hbox{and}\quad \|f\|_{C^{1+\beta}}\leq 1
$$
for some $\beta\in(0,1)$. Note that the singularity in $\EE$ is in fact removable owing to the symmetry of the disc. 
Indeed,  following  standard computations we get the expansion
\begin{equation}\label{Expa1}
I_\EE(\phi_\EE(w))=-\frac1\EE \overline{w}+J_\EE(\phi_\EE(w))
\end{equation}
where $J_\EE$ belongs to the space $C^\beta$ and can be extended for $\EE\in (-\EE_0,\EE_0)$ with $\EE_0>0.$   Setting
$$
G(\EE,\Omega,f(w))\equiv\textnormal{Im}\bigg\{ \Big(2\Omega\big(\EE{\phi_\EE(w)}-d)+I_\EE\big(\phi_\EE(w)\big)\Big)\, w\, \phi_\EE^\prime(w)\bigg\},
$$
then the equation of the vortex pairs is  simply given by
$$
\forall \, w\in \mathbb{T},\quad G(\EE,\Omega,f(w))=0.
$$
It follows from the expansion \eqref{Expa1} that we can get rid of  the singularity in $\EE$ and this is the first step towards the application of 
the implicit function theorem. Before giving further details we should first  fix the function spaces. Let 
$$
X^0=\big\{ f\in C^{1+\beta}(\mathbb{T}), f(w)=\sum_{n\geq1} a_n w^{-n}\big\}, 
$$
and 
$${Y}^0=\Big\{ f\in C^{\beta}(\mathbb{T}), f=\sum_{n\geq1} a_n e_n,\, a_n\in \RR\Big\},\quad \widehat{Y^0}=\big\{f\in Y^0, a_1=0\big\},\, e_n(w)\equiv\hbox{Im}(w^n).
$$
According to Proposition \ref{prop1} the function $G:(-\frac12,\frac12)\times \RR\times B_1^0\to Y^0$ is well-defined and it is  of class $C^1$, where  $B_1^0$ is the open unit ball of $X^0$.  Moreover 
$$
\partial_f G(0,\Omega, 0) h(w)=-\hbox{Im}(h^\prime(w)).
$$
However, this operator is not invertible from $X^0$ to $Y^0$ but it does from $X^0$ to $\widehat{Y^0}.$ The next step is to choose carefully $\Omega$ such that the image of the nonlinear functional $G$ is contained in the vector \mbox{space $\widehat{Y^0}.$} This will be done carefully in Section \ref{real-velo} and  leads eventually  to a new nonlinear constraint of the type $\Omega=\Omega(\EE,f).$ Consequently the equation of the vortex pairs becomes
$$
F(\EE,f(w))\equiv G(\EE,\Omega(\EE,f),f)=0.
$$
Note that with this formulation  the point vortex configuration corresponds to $F(0,0)=0$ in which case  $\Omega=\Omega_{sing}^0=\frac{1}{4\pi d^2}$. In addition, from the platitude of $\Omega$ we deduce that the linearized operator remains the same, that is, 
$$\partial_f F(0,0)=\partial_f G(0,\Omega_{sing}^0, 0)
$$ 
which is invertible from $X^0$ to $\widehat{Y^0}$. Therefore and at  this stage one can use 
the implicit function theorem which implies the local  existence of  
a unique curve of solutions $\EE\mapsto \phi_\EE$ passing through $(0,0)$ and remark  
that each point of this curve is a nontrivial  corotating vortex pair of symmetric simply connected  patches.

The remaining of the paper is organized as follows. In Section \ref{prelim} we shall gather some tools dealing with the function spaces  and give some results on Newton and Riesz potentials. In Section \ref{st-vor} we shall write down the equations governing the corotating and translating vortex pairs of symmetric patches for both Euler and $(\hbox{SQG})_\alpha$ equations. Sections \ref{co-E} and \ref{count-E} are dedicated to the proofs of the Main Theorem.

\vspace{0,2cm}

{\bf {Notation.}}
We need to fix some notations that will  be frequently used along this paper.
 We denote by $C$ any positive constant that may change from line to line.
 We denote by $\mathbb{D}$ the unit disc and its boundary, the unit circle, is denoted by  $\mathbb{T}$. 
  Let $f:\mathbb{T}\to \CC$ be a continuous function, we define  its  mean value by,
$$
\fint_{\mathbb{T}} f(\tau)d\tau\equiv \frac{1}{2i\pi}\int_{\mathbb{T}}  f(\tau)d\tau,
$$
where $d\tau$ stands for the complex integration. Finally, for $x\in \RR$ and $n\in \NN$,  we use the notation $(x)_n$ to denote  the  Pochhammer symbol defined by,
$$
(x)_n = \begin{cases}   1   & n = 0 \\
  x(x+1) \cdots (x+n-1) & n \geq1.
 \end{cases}
 $$

\section{Preliminaries and background}\label{prelim}



In this section we shall  briefly recall the classical H\"{o}lder spaces on the periodic case and 
state some classical facts on the continuity of fractional integrals over these spaces.
 It is convenient to think of  $2\pi-$periodic function $f:\RR\to\CC$ as a function of the complex 
 variable $w=e^{i\eta}$ rather than a function of the real variable $\eta.$ To be more precise, 
 let  $f:\mathbb{T}\to \RR^2$, be a continuous function, then it  can be assimilated to  a $2\pi-$periodic function $g:\RR\to\RR$ via the relation
$$
f(w)=g(\eta),\quad w=e^{i\eta}.
$$
Hence when  $f$ is  smooth enough we get
$$
f^\prime(w)\equiv\frac{df}{dw}=-ie^{-i\eta}g'(\eta).
$$  
Because  $d/dw$ and $d/d\eta$ differ only by a smooth factor with modulus one  we shall  in the sequel work with $d/dw$ instead of $d/d\eta$ which appears to be  more convenient in the computations.
Now we shall introduce  H\"older spaces  on the unit circle $\mathbb{T}$.
\begin{definition}
Let 
$0<\beta<1$. We denote by $C^\beta(\mathbb{T}) $  the space of continuous functions $f$ such that
$$
\Vert f\Vert_{C^\beta(\mathbb{T})}\equiv \Vert f\Vert_{L^\infty(\mathbb{T})}+\sup_{x\neq y\in \mathbb{T}}\frac{\vert f(x)-f(y)\vert}{\vert x-y\vert^\beta}<\infty.
$$
For any integer $n$ the space $C^{n+\beta}(\mathbb{T})$ stands for the set of functions $f$ of class $C^n$ whose $n-$th order derivatives are H\"older continuous  with exponent $\beta$. This space is equipped with the usual  norm,
$$
\Vert f\Vert_{C^{n+\beta}(\mathbb{T})}\equiv \Vert f\Vert_{L^\infty(\mathbb{T})}+\Big\Vert \frac{d^n f}{dw^n}\Big\Vert_{C^\beta(\mathbb{T})}.
$$ 
\end{definition}
Recall that the Lipschitz (semi)-norm is defined as follows.
$$
\|f\|_{\textnormal{Lip}(\mathbb{T})}=\sup_{x\neq y}\frac{|f(x)-f(y)|}{|x-y|}.
$$
Now we list some classical properties that will be used later in several sections.
\begin{enumerate}
\item For $n\in \mathbb{N}, \beta\in ]0,1[$ the space $C^{n+\beta}(\mathbb{T})$ is an algebra.
\item For $K\in L^1(\mathbb{T})$ and $f\in C^{n+\beta}(\mathbb{T})$ we have the convolution law,
$$
\|K*f\|_{C^{n+\beta}(\mathbb{T})}\le \|K\|_{L^1(\mathbb{T})}\|f\|_{C^{n+\beta}(\mathbb{T})}.
$$

\end{enumerate}

The next result is used frequently and it deals with fractional integrals of the 
following type,
\begin{equation}\label{opsin}
\mathcal{T}(f)(w)=\int_{\mathbb{T}}K(w,\tau)\, f(\tau)d\tau,
\end{equation}
with $K:\mathbb{T}\times\mathbb{T}\to \mathbb{C}$ being  a singular kernel satisfying some properties. The problem on the smoothness of this operator will appear 
naturally when we shall deal with the regularity of the nonlinear functional defining steady  vortex pairs. 
 The result  that we shall discuss with respect to this  subject  is classical and whose proof can be found for instance in  \cite{H-H,MOV}. 
\begin{lemma}\label{noyau}
Let $0\leq \alpha< 1$ and consider a function $K:\mathbb{T}\times\mathbb{T}\to \mathbb{C}$ with the following properties.
 There exits $C_0>0$ such that,
\begin{enumerate}
\item $K$ is measurable on $\mathbb{T}\times\mathbb{T}\backslash\{(w,w),\, w\in \mathbb{T}\}$ and
$$
\big\vert K(w,\tau)\big\vert\leq \frac{C_0}{\vert w-\tau\vert^\alpha}, \quad\forall\,  w\neq \tau\in \mathbb{T}.
$$
\item For each $\tau\in \mathbb{T}$, $w\mapsto K(w,\tau)$ is differentiable in $\mathbb{T}\backslash\{\tau\}$ and 
$$
  \big\vert\partial_w K(w,\tau)\big\vert\leq \frac{C_0}{\vert w-\tau\vert^{1+\alpha}}, \quad \forall\, w\neq \tau\in \mathbb{T}.
$$
\end{enumerate}
Then 
\begin{enumerate}[A)]

 \item The operator $\mathcal{T} $ defined by \eqref{opsin} is continuous from $L^\infty(\mathbb{T})$ 
to $C^{1-\alpha}(\mathbb{T})$. 
More precisely, there exists a constant $C_\alpha$ depending only on $\alpha$ such that
$$
\Vert \mathcal{T}(f)\Vert_{1-\alpha}\leq C_\alpha C_0\Vert f\Vert_{L^\infty}.
$$
\item For $\alpha=0$ the operator $\mathcal{T} $  is continuous from $L^\infty(\mathbb{T})$ 
to $C^{\beta}(\mathbb{T})$ for any $0<\beta<1$. That is, there exists a constant $C_\beta$ depending only on $\beta$ such that
$$
\Vert \mathcal{T}(f)\Vert_{\beta}\leq C_\beta C_0\Vert f\Vert_{L^\infty}.
$$
\end{enumerate}
\end{lemma}

As a by-product we obtain a result that will be frequently used through this paper.
\begin{coro}\label{cor}
Let $0< \alpha<1$, $\phi:\mathbb{T}\to \phi(\mathbb{T})$ be a bi-Lipschitz function with real Fourier coefficients  and define the operator
$$
\mathcal{T}_\phi: f\mapsto\displaystyle{\mathop{{\fint}_{\mathbb{T}}}}\frac{f(\tau)}{\vert \phi(w)-\phi(\tau)\vert^\alpha}d\tau,\quad w\in \mathbb{T}.
$$
Then  $\mathcal{T}_\phi:L^\infty\big(\mathbb{T}\big)\to C^{1-\alpha}\big(\mathbb{T}\big)$ is continuous
with the estimate
$$
\Vert \mathcal{T}_\phi(f)\Vert_{C^{1-\alpha}(\mathbb{T})}\leq C\Big(\|\phi^{-1}\|_{\textnormal{Lip}(\mathbb{T})}^\alpha+\|\phi\|_{\textnormal{Lip}(\mathbb{T})}^2\|\phi^{-1}\|_{\textnormal{Lip}(\mathbb{T})}^{1+\alpha}\Big)\Vert f\Vert_{L^\infty(\mathbb{T})},
$$
where $C$ is a positive constant depending only on $\alpha$.
\end{coro}

\section{Steady vortex pairs models}\label{st-vor}
The aim of this section is to derive the equations governing co-rotating and translating symmetric 
pairs of patches. In the first step, we shall write down the equations for the rotating pairs for 
Euler and $(\hbox{SQG})_\alpha$ equations. In the second step, we shall be concerned with the counter-rotating 
vortex pairs sometimes called  translating pairs.  Notice that we prefer to use the conformal 
parametrization because it  is more convenient in the computations especially through its holomorphic structure.
\subsection{Corotating vortex pairs}\label{SEc-rotat11}
Let $D_1$ be  a bounded simply connected domain containing the origin and contained in the  ball $B(0,2).$ 
For  $\varepsilon\in]0,1[$  and $d>2$ we  define the domains
$$
D_1^\varepsilon=\varepsilon D_1\quad \hbox{and}\quad D_2^\varepsilon=-D_1^\varepsilon+ 2d.
$$
Set
$$\theta_{0,\EE}=\frac{1}{\varepsilon^2}\chi_{D_1^\varepsilon}+ \frac{1}{\varepsilon^2} \chi_{D_2^\varepsilon}
$$ 
and assume that this gives rise to a rotating pairs of patches  about the centroid of the system $(d,0)$ 
and with an angular velocity $\Omega$. 
According to (\cite{H-F-M-V}, p.1896) this condition holds true if and only if   
\begin{equation}\label{V1}
\operatorname{Re}(-i \,\Omega \,\big(\overline{z}-d\big)\, \vec{n}(z)) =
\textnormal{Re}(\overline{v(z)}\, \vec{n}), \quad\forall  z \in \partial D_1^\EE\cup \partial D_2^\EE,
\end{equation}
where $\vec{n}(z)$ is the exterior unit normal vector to the boundary
of $D_1^\EE\cup D_2^\EE$ at the point $z.$  Next we shall discuss separately   Euler equations  and the case $\alpha\in (0,1)$ due to the difference structures of their Green functions.  

\subsubsection{Euler equations} \label{mod-Euler}
It is well-known that the  
velocity can be recovered for the vorticity according to 
 Biot-Savart law,
$$
\overline{v(z)} = -\frac{i}{2 \pi \, \EE^2} \int_{D_1^\EE}
\frac{dA(\zeta)}{z-\zeta}-\frac{i}{2 \pi\EE^2} \int_{D_2^\EE}
\frac{dA(\zeta)}{z-\zeta}, \quad \forall z \in \mathbb{C}.
$$
From Green-Stokes formula we record that
$$
-\frac{1}{\pi} \int_{D} \frac{dA(\zeta)}{z-\zeta} =
 \fint_{ \partial D} \frac{\overline{\xi}-
\overline{z}}{\xi-z}\, d\xi, \quad \forall z \in \mathbb{C}.
$$
 Therefore
 \begin{equation}\label{rotation0}
 \textnormal{Re}\Big\{ \Big(2\Omega\big(\overline{z}-d)+I(z)\Big)\, \vec{\tau}\Big\}=0,\quad\forall  z\in \partial D_1^\varepsilon\cup\partial D_2^\varepsilon ,
\end{equation}
with $\vec{\tau}$ being the unit tangent vector to $ \partial D_1^\EE\cup \partial D_2^\EE$
positively oriented and 
$$
I(z)=\frac{1}{\varepsilon^2}\fint_{\partial D_1^\varepsilon}\frac{\overline{\xi}-\overline{z}}{\xi-z}d\xi+\frac{1}{\varepsilon^2}
\fint_{\partial D_2^\varepsilon}\frac{\overline{\xi}-\overline{z}}{\xi-z}d\xi.
$$
Changing in the last integral $\xi$ to $-\xi+2d$, which sends $\partial D_2^\varepsilon $ to $\partial D_1^\varepsilon,$ we get
$$
I(z)=\frac{1}{\varepsilon^2}\fint_{\partial D_1^\varepsilon}\frac{\overline{\xi}-\overline{z}}{\xi-z}d\xi-
\frac{1}{\varepsilon^2}\fint_{\partial D_1^\varepsilon}\frac{\overline{\xi}+\overline{z}-2d}{\xi+z-2d}d\xi.
$$
We can check that if the equation \eqref{rotation0} is satisfied for all $ z\in \partial D_1^\varepsilon,$ 
then it will be surely  satisfied for all $ z\in \partial D_2^\varepsilon$. This follows easily  from the identity
$$
I(-z+2d)=-I(z).
$$
Now observe that when $z\in \partial D_1^\varepsilon$ then $-z+2d\notin \overline{D_1^\varepsilon}$ and thus  residue theorem allows to get
$$
I(z)=\frac{1}{\varepsilon^2}\fint_{\partial D_1^\varepsilon}\frac{\overline{\xi}-\overline{z}}{\xi-z}d\xi-\frac{1}{\varepsilon^2}\fint_{\partial D_1^\varepsilon}\frac{\overline{\xi}}{\xi+z-2d}d\xi.
$$
Denote $\Gamma_1=\partial D_1$ then by the  change of variables $\xi\mapsto \varepsilon \xi$ and $z\mapsto \varepsilon z$ the equation \eqref{rotation0} becomes
\begin{equation*}
 \textnormal{Re}\Big\{ \Big(2\Omega\big(\varepsilon\overline{z}-d)+I_\varepsilon(z)\Big)\, \vec{\tau}\Big\}=0,\quad\forall  z\in\Gamma_1.
 \end{equation*}
with
\begin{eqnarray*}
I_\varepsilon(z)&\equiv&I(\varepsilon z)\\
&=&\frac{1}{\varepsilon}\fint_{\Gamma_1}\frac{\overline{\xi}-\overline{z}}{\xi-z}d\xi-\fint_{\Gamma_1}\frac{\overline{\xi}}{\EE\xi+\EE z-2d}d\xi\\
&\equiv&I_\EE^1(z)-I_\EE^2(z).
\end{eqnarray*}
We shall search for domains $D_1$ which are small perturbations of the unit disc with an amplitude of order $\EE.$ More precisely, we shall in   the conformal parametrization   $\phi: \mathbb{T}\to \partial D_1$ look for a solution in the form
\begin{equation*}
\phi(w)=w+\EE f(w),\hbox{with}\quad f(w)=\sum_{n\geq 1}\frac{a_n}{w^n},\quad  a_n\in\RR.
\end{equation*}
We remark that  the assumption $a_n\in \RR$ means that the domain $D_1$ is symmetric with respect to the real axis.
Setting $z=\phi(w)$, then for $w\in \mathbb{T}$ a tangent vector to the boundary at the point $z$ is given by
$$
\vec{\tau}=i\,{w}\, {{\phi'(w)}}=i\,w\big(1+\EE f^\prime(w)\big).
$$
Thus  the steady vortex pairs  equation becomes
\begin{equation}\label{rotation2}
 \textnormal{Im}\Big\{ \Big(2\Omega\Big[\varepsilon\overline{w}+\EE^2 f(\overline{w})-d\Big]+I_\varepsilon(\phi(w))\Big)\, w\big(1+\EE f^\prime(w)\big)\Big\}=0,\quad\forall w\in \mathbb{T}.
 \end{equation}
 Notice that we have used that $f$ has real Fourier coefficients and thus $\overline{f(w)}=f(\overline{w})$. 
By using the notation $A=\tau-w$ and $B=f(\tau)-f(w)$ we can write   for all $w\in \mathbb{T}$
\begin{eqnarray*}
I_\varepsilon^1(\phi(w))
&=&\frac{1}{\varepsilon}\fint_{\mathbb{T}}\frac{\overline{\tau}-\overline{w}+\EE\big(f(\overline{\tau})-f(\overline{w})\big)}{\tau-w+\EE(f(\tau)-f(w))}\big(1+\EE f^\prime(\tau)\big)d\tau\\
&=&\fint_{\mathbb{T}}\frac{\overline{A}+\EE\overline{B}}{A+\EE B} f^\prime(\tau)d\tau
+\fint_{\mathbb{T}}\frac{A\overline{B}-\overline{A} B}{A(A+\EE B)}d\tau+\frac1\EE\fint_{\mathbb{T}}\frac{\overline{A}}{A}d\tau\\
&=&\fint_{\mathbb{T}}\frac{\overline{A}+\EE\overline{B}}{A+\EE B} f^\prime(\tau)d\tau
+\fint_{\mathbb{T}}\frac{A\overline{B}-\overline{A} B}{A(A+\EE B)}d\tau-\frac1\EE \overline{w},
\end{eqnarray*}
where we have used  the obvious formula
\begin{eqnarray*}
\fint_{\mathbb{T}}\frac{\overline{A}}{A}d\tau&=&-\overline{w}\fint_{\mathbb{T}} \frac{d\tau}{\tau}\\
&=&-\overline{w}.
\end{eqnarray*}
This leads to a significant   cancellation and  the singular term will disappear from  the full nonlinearity due in particular to the symmetry of the disc,\begin{eqnarray*}
 \textnormal{Im}\Big\{ I_\varepsilon^1(\phi(w))\, w\big(1+\EE f^\prime(w)\big)\Big\}&=& \textnormal{Im}\bigg\{ \Big(\fint_{\mathbb{T}}\frac{\overline{A}+\EE\overline{B}}{A+\EE B} f^\prime(\tau)d\tau+\fint_{\mathbb{T}}\frac{A\overline{B}-\overline{A} B}{A(A+\EE B)}d\tau\Big)w\big[1+\EE f^\prime(w)\big]\bigg\}\\
 &
&- \textnormal{Im}(f^\prime(w)),\quad\forall w\in \mathbb{T}.
 \end{eqnarray*}
For the second term $I_\EE^2(\phi(w)$ it takes the form
$$
I_\EE^2(\phi(w)=\fint_{\mathbb{T}}\frac{\overline{(\tau}+\EE f(\overline{\tau}))(1+\EE f^\prime(\tau))}
{\EE(\tau+w)+\EE^2\big(f(\tau)+f(w)\big)-2d}d\tau.
$$
Hence the steady vortex pairs  equation is equivalent to
\begin{equation}\label{def1}
 -2G^0(\EE, \Omega,f)\equiv\textnormal{Im}(F^0(\EE, \Omega,f))=0
\end{equation}
with
\begin{eqnarray*}
F^0(\EE, \Omega, f(w))&=& 2\Omega \Big(\EE \overline{w}+\EE^2f(\overline{w})-d\Big) w\big(1+\EE f^\prime(w)\big)-f^\prime(w)\\
&+&\Big(\fint_{\mathbb{T}}\frac{\overline{A}+\EE\overline{B}}{A+\EE B} f^\prime(\tau)d\tau+\fint_{\mathbb{T}}\frac{A\overline{B}-\overline{A} B}{A(A+\EE B)}d\tau\Big)w\big(1+\EE f^\prime(w)\big)\\
&-&\bigg(\fint_{\mathbb{T}}\frac{\big(\overline{\tau}+\EE{f(\overline\tau)}\big)\big(1+\EE f^\prime(\tau)\big)}{\EE(\tau+w)+\EE^2 \big(f(\tau)+f(w)\big)-2d}d\tau\bigg)w\big(1+\EE f^\prime(w)\big)\\
&\equiv& F_1(\EE,\Omega,f(w))+F_2(\EE,f(w))+F_3(\EE,f(w)).
\end{eqnarray*}
We point out that we have added a factor $-2$ in the definition of $G^0$ given by \eqref{def1} in order to unify the notation with the function $G^\alpha$ that we shall introduce in next section for \mbox{the $(\hbox{SQG})_\alpha$.}
\subsubsection{$(\hbox{SQG})_\alpha$ equations.}
First we remark that the equation \eqref{V1} can be written in the form,
 \begin{equation}\label{rotsqgg1}
\Omega \,\textnormal{Re}\big\{ \big(z-d)\, \overline{\vec{\tau}}\big\}=\textnormal{Im}\big\{v(z)\,\overline{\vec{\tau}}\big\},
\quad \forall z\in \partial D_1^\EE\cup\partial D_2^\EE,
\end{equation}
where as before $\vec{\tau}$ denotes  a tangent vector to the boundary  at the point $z.$  This equation is equivalent to \begin{equation*}
\,\textnormal{Re}\Big\{ \Big(\Omega \big( z-d)+iv(z)\Big)\, \overline{\vec{\tau}}\Big\}=0,\quad \forall z\in \partial D_1^\EE \cup\partial D_2^\EE.
\end{equation*} The velocity  can be recovered from the boundary as follows, see for instance \cite{H-H}, 
\begin{equation*}
v(z)=\frac{C_\alpha}{2\pi \,\EE^2}{\int}_{\partial D_1^\EE}\frac{1}{\vert z-\xi\vert^\alpha}d\xi+\frac{C_\alpha }{2\pi\,\EE^2}
{\int}_{\partial D_2^\EE}\frac{1}{\vert z-\xi\vert^\alpha}d\xi,\quad \forall z\in \mathbb{C}.
\end{equation*}
Using in the last integral the change of variables  $\xi\mapsto -\xi+2d$,  we deduce that
\begin{equation}\label{veltu}
v(z)=\frac{C_\alpha}{2\pi \,\EE^2}{\int}_{\partial D_1^\EE}\frac{1}{\vert z-\xi\vert^\alpha}d\xi-\frac{C_\alpha }{2\pi\,\EE^2}
{\int}_{\partial D_1^\EE}\frac{1}{\vert z+\xi-2d\vert^\alpha}d\xi,\quad \forall z\in \mathbb{C}.
\end{equation}
%
We point out that by a symmetry argument  if  the equation \eqref{rotsqgg1} is satisfied for all $ z\in \partial D_1^\varepsilon$ then it will be  also satisfied 
for all $ z\in \partial D_2^\varepsilon$. This follows from the identity
$$
v(-z+2d)=-v(z).
$$
As $D_1^\EE=\EE D_1$ then  using a change of variable the equation becomes
\begin{equation}\label{rotsqg1}
\,\textnormal{Re}\Big\{ \Big(\Omega \big(\EE z-d)+I_\EE(z)\Big)\, \overline{\vec{\tau}}\Big\}=0,\quad \forall z\in \partial D_1,
\end{equation}
with
$$
I_\EE(z)=-\frac{C_\alpha}{\,\EE^{1+\alpha}}\fint_{\partial D_1}\frac{1}{\vert z-\xi\vert^\alpha}d\xi+\frac{C_\alpha }{\EE}
\fint_{\partial D_1}\frac{1}{\vert\EE z+\EE \xi-2d\vert^\alpha}d\xi.
$$
We shall look for the domains $D_1$ which are small perturbation of the unit disc with an amplitude of order $\EE^{1+\alpha}.$ More precisely, we shall in   the conformal parametrization   $\phi: \mathbb{T}\to \partial D_1$ look for a solution in the form
\begin{eqnarray*}
\phi(w)&=&w+\EE^{1+\alpha} f(w)\\
&=&w+\EE^{1+\alpha}\sum_{n\geq 1}\frac{a_n}{w^n},\quad a_n\in \RR.
\end{eqnarray*}
For $w\in \mathbb{T}$ the conjugate of a tangent vector is given by
$
\overline{\vec{\tau}}=-i\,\overline{w}\, {\overline{\phi^\prime(w)}}
$ and 
therefore for any $ w\in \mathbb{T},$
\begin{eqnarray}\label{model}
\nonumber G^\alpha(\EE,\Omega,f(w))&\equiv&\textnormal{Im}\Bigg\{\bigg(\Omega\Big[\EE w+\EE^{2+\alpha}f(w)-d\Big]+
I(\EE,f(w))\bigg)\, \overline{w}\,\Big(1+\EE^{1+\alpha}\overline{f^\prime(w)}\Big)\Bigg\}\\
&=&\textnormal{Im}\Big({F^\alpha}(\EE,\Omega,f(w)\Big)=0, 
\end{eqnarray}
with
\begin{eqnarray}\label{Form10}
I(\EE,f(w))
\nonumber &=&-\frac{C_\alpha}{\EE^{1+\alpha}}\mathop{{\fint}}_\mathbb{T}\frac{\phi'(\tau)d\tau}{\vert \phi(w)-\phi(\tau)\vert^\alpha}
+\frac{C_\alpha}{\EE}\mathop{{\fint}}_\mathbb{T}\frac{\phi'(\tau)d\tau}{\vert \EE\phi(w)+\EE \phi(\tau)-2d\vert^\alpha}\\
&\equiv&-I_1(\EE,f(w))+I_2(\EE,f(w)).
\end{eqnarray}
We shall split $G$ into three terms
\begin{equation}\label{split1}
G^\alpha=G_1-G_2+G_3
\end{equation}
with
$$
G_1(\EE,\Omega,f(w))=\textnormal{Im}\Bigg\{\Omega\Big[\EE w+\EE^{2+\alpha}f(w)-d\Big]\, \overline{w}\,\Big(1+\EE^{1+\alpha}\overline{f^\prime(w)}\Big)\Bigg\},
$$
$$
G_2(\EE,f(w))=\textnormal{Im}\Big\{I_1(\EE,f(w))\, \overline{w}\,\Big(1+\EE^{1+\alpha}\overline{f^\prime(w)}\Big)\Big\},
$$
and 
$$
G_3(\EE,f(w))=\textnormal{Im}\Big\{I_2(\EE,f(w))\, \overline{w}\,\Big(1+\EE^{1+\alpha}\overline{f^\prime(w)}\Big)\Big\}.
$$
 \subsection{Counter-rotating vortex pairs}\label{Sec-HM}
 
 As for the corotating pairs  we shall distinguish between Euler equations and the case $0<\alpha<1.$ As before let $D_1$ be a bounded domain containing the origin and contained in the  ball $B(0,2).$
  For  $\varepsilon\in]0,1[$ and $d>2$ we  define 
$$
D_1^\varepsilon=\varepsilon D_1\quad \hbox{and}\quad D_2^\varepsilon=-D_1^\varepsilon+ 2d.
$$
Set 
$$\theta_{0}=\frac{1}{\varepsilon^2}\chi_{D_1^\varepsilon}- \frac{1}{\varepsilon^2} \chi_{D_2^\varepsilon}
$$ 
and assume that $\theta_0$ travels steadily in the $(Oy)$ direction with uniform velocity $U.$ Then in the moving frame the pair of the  patches  is stationary and consequently  the analogous of the equation \eqref{V1} is
\begin{equation}\label{V-count}
\textnormal{Re}\big\{\big(\overline{v(z)}+iU\big)\, \vec{n}\big\}=0,\quad  \forall z\in \partial D_1^{\EE}\cup\partial D_2^{\EE}.
\end{equation}
\subsubsection{Euler equations} 
One has from \eqref{V-count}
\begin{equation}\label{V004}
\textnormal{Re} \left\{ \left(2 U+ I(z)
 \right)\,\vec{\tau} \right\} = 0 , \quad
 \forall z\in \partial D_1^{\EE}\cup\partial D_2^{\EE},
\end{equation}
with
$$
I(z)=\frac{1}{\varepsilon^2}\fint_{\partial D_1^\varepsilon}\frac{\overline{\xi}-\overline{z}}{\xi-z}d\xi-\frac{1}{\varepsilon^2}\fint_{\partial D_2^\varepsilon}\frac{\overline{\xi}-\overline{z}}{\xi-z}d\xi.
$$
In the last integral changing $\xi$ to $-\xi+2d$ which sends $\partial D_2^\varepsilon $ to $\partial D_1^\varepsilon $ we get
$$
I(z)=\frac{1}{\varepsilon^2}\fint_{\partial D_1^\varepsilon}\frac{\overline{\xi}-\overline{z}}{\xi-z}d\xi+\frac{1}{\varepsilon^2}\fint_{\partial D_1^\varepsilon}\frac{\overline{\xi}+\overline{z}-2d}{\xi+z-2d}d\xi.
$$
As for the corotating case, using the identity 
$$
I(-z+2d)=I(z)
$$
one can check that if the equation \eqref{V004} is satisfied for
all $ z\in \partial D_1^\varepsilon$ then it is also satisfied for all $ z\in \partial D_2^\varepsilon$.
\\
Now observe that when $z\in \partial D_1^\varepsilon$ then $-z+2d\notin \overline{D_1^\varepsilon}$ and using residue theorem we obtain
$$
I(z)=\frac{1}{\varepsilon^2}\fint_{\partial D_1^\varepsilon}\frac{\overline{\xi}-\overline{z}}{\xi-z}d\xi+\frac{1}{\varepsilon^2}\fint_{\partial D_1^\varepsilon}\frac{\overline{\xi}}{\xi+z-2d}d\xi.
$$
Let $\Gamma_1=\partial D_1$ then by change of variables $\xi\to \varepsilon \xi$ and $z\to \varepsilon z$. The equation \eqref{V004} becomes
\begin{equation*}
 \textnormal{Re}\Big\{ \Big(2U+I_\varepsilon(z)\Big)\,\vec{\tau}\Big\}=0,\quad\forall  z\in\Gamma_1,
 \end{equation*}
with
\begin{eqnarray*}
I_\varepsilon(z)&=&I(\varepsilon z)\\
&=&\frac{1}{\varepsilon}\fint_{\Gamma_1}\frac{\overline{\xi}-\overline{z}}{\xi-z}d\xi+\fint_{\Gamma_1}\frac{\overline{\xi}}{\EE\xi+\EE z-2d}d\xi\\
&\equiv&I_\EE^1(z)+I_\EE^2(z).
\end{eqnarray*}
we shall now use the conformal parametrization of the boundary $\Gamma_1$,
\begin{equation*}
\phi(w)=w+\EE f(w),\hbox{with}\quad f(w)=\sum_{n\geq 1}\frac{a_n}{w^n}, a_n\in\RR.
\end{equation*}
Setting $z=\phi(w)$ and $\xi=\phi(\tau)$, then for $w\in \mathbb{T}$ a tangent vector  at the point $\phi(w)$ is given by
$$
\vec{\tau}=i{w}\, {{\phi'(w)}}=iw\big(1+\EE f^\prime(w)\big).
$$
The V-states equation becomes
\begin{equation*}
 \textnormal{Im}\Big\{ \Big(2U+I_\varepsilon(\phi(w))\Big)\, w\big(1+\EE f^\prime(w)\big)\Big\}=0,\quad\forall w\in \mathbb{T}.
 \end{equation*}
As in the rotating case, with  the notation $A=\tau-w$ and $B=f(\tau)-f(w)$ we get  for  $w\in \mathbb{T}$
\begin{eqnarray*}
I_\varepsilon^1(\phi(w))
&=&\fint_{\mathbb{T}}\frac{\overline{A}+\EE\overline{B}}{A+\EE B} f^\prime(\tau)d\tau
+\fint_{\mathbb{T}}\frac{A\overline{B}-\overline{A} B}{A(A+\EE B)}d\tau-\frac1\EE \overline{w}.
\end{eqnarray*}
This yields
\begin{eqnarray*}
 \textnormal{Im}\Big\{ I_\varepsilon^1(\phi(w))\, w\big(1+\EE f^\prime(w)\big)\Big\}
&=& \textnormal{Im}\Big\{ \Big(\fint_{\mathbb{T}}\frac{\overline{A}+\EE\overline{B}}{A+\EE B} f^\prime(\tau)d\tau+\fint_{\mathbb{T}}\frac{A\overline{B}-\overline{A} B}{A(A+\EE B)}d\tau\Big)w\big(1+\EE f^\prime(w)\big)\Big\}\\
 &
&- \textnormal{Im}(f^\prime(w)),\quad\forall w\in \mathbb{T}.
 \end{eqnarray*}
The second term $I_\EE^2(\phi(w)$
takes the form
$$
I_\EE^2(\phi(w))=\fint_{\mathbb{T}}\frac{(\overline{\tau}+\EE\overline{f(\tau)})(1+\EE f^\prime(\tau))}{\EE(\tau+w)+\EE^2\big(f(\tau)+f(w)\big)-2d}d\tau.
$$
Hence the V-states equation becomes
\begin{equation}\label{DDX1}
 -2G^0(U,\EE,f)\equiv\textnormal{Im}(F^0(U,\EE, f)=0
\end{equation}
with
\begin{eqnarray*}
F^0(U,\EE, f(w))&=& 2U w\big(1+\EE f^\prime(w)\big)-f^\prime(w)\\
&+&\bigg(\fint_{\mathbb{T}}\frac{\overline{A}+\EE\overline{B}}{A+\EE B} f^\prime(\tau)d\tau+\fint_{\mathbb{T}}\frac{A\overline{B}-\overline{A} B}{A(A+\EE B)}d\tau\bigg)w\big(1+\EE f^\prime(w)\big)\\
&+&\bigg(\fint_{\mathbb{T}}\frac{\overline{\tau}+\EE{f(\overline\tau)}}{\EE(\tau+w)+\EE^2 \big(f(\tau)+f(w)\big)-2d}\big(1+\EE f^\prime(\tau)\big)d\tau\bigg)w\big(1+\EE f^\prime(w)\big)\\
&\equiv& F_1(U,\EE,f(w))+F_2(\EE,f(w))+F_3(\EE,f(w)).
\end{eqnarray*}
For the same reason as in the rotating case we add the factor $-2$ in \eqref{DDX1} in order to unify the expression with the function $G^\alpha$ that will appear later for the $(\hbox{SQG})_\alpha$.

\subsubsection{Case $\alpha\in(0,1)$} 

%
The equation \eqref{V-count} can be written in the form
$$
\textnormal{Re}\Big\{\big({v(z)}-iU\big)\, \overline{\vec{n}}\Big\}=0 ,\quad  \forall z\in \partial D_1^{\EE}\cup\partial D_2^{\EE}.
$$
The velocity associated to  this model is
\begin{equation*}
v(z)=\frac{C_\alpha}{2\pi \,\EE^2}{\int}_{\partial D_1^\EE}\frac{1}{\vert z-\xi\vert^\alpha}d\xi-\frac{C_\alpha }{2\pi\,\EE^2}{\int}_{\partial D_2^\EE}\frac{1}{\vert z-\xi\vert^\alpha}d\xi,\quad \forall z\in \mathbb{C}.
\end{equation*}
Changing $\xi$ to $-\xi+2d$ in the last integral we get
\begin{equation}\label{veltu1}
v(z)=\frac{C_\alpha}{2\pi \,\EE^2}{\int}_{\partial D_1^\EE}\frac{1}{\vert z-\xi\vert^\alpha}d\xi+\frac{C_\alpha }{2\pi\,\EE^2}{\int}_{\partial D_1^\EE}\frac{1}{\vert z+\xi-2d\vert^\alpha}d\xi,\quad \forall z\in \mathbb{C}.
\end{equation}
Therefore the V-states equation be can be written in the form
\begin{equation}\label{rotsqX1}
\,\textnormal{Re}\Big\{ \Big(-U+I_\EE(z)\Big)\, \overline{\vec{\tau}}\Big\}=0,\quad \forall z\in \partial D_1
\end{equation}
with
$$
I_\EE(z)\equiv\frac{C_\alpha}{\,\EE^{1+\alpha}}\fint_{\partial D_1}\frac{1}{\vert z-\xi\vert^\alpha}d\xi+\frac{C_\alpha }{\EE}\fint_{\partial D_1}\frac{1}{\vert\EE z+\EE \xi-2d\vert^\alpha}d\xi.
$$
Using the conformal parametrization,\begin{eqnarray*}
\phi(w)&=&w+\EE^{1+\alpha} f(w)\\
&\equiv&w+\EE^{1+\alpha}\sum_{n\geq 1}\frac{a_n}{w^n}.
\end{eqnarray*}
For $w\in \mathbb{T}$ the conjugate of a tangent vector is given by
$$
\overline{z'}=-i\overline{w}\, {\overline{\phi'(w)}}.
$$
Therefore for any $ w\in \mathbb{T},$
\begin{equation}\label{modelX}
G^\alpha(\EE,\Omega,f(w))\equiv\textnormal{Im}\bigg\{\Big(-U+I\big(\EE,f(w)\big)\Big)\, \overline{w}\,\Big(1+\EE^{1+\alpha}\overline{f^\prime(w)}\Big)\bigg\}=0, 
\end{equation}
with
\begin{eqnarray}\label{Form10X}
I(\EE,f(w))
\nonumber &=&\frac{C_\alpha}{\EE^{1+\alpha}}\mathop{{\fint}}_\mathbb{T}\frac{\phi'(\tau)d\tau}{\vert \phi(w)-\phi(\tau)\vert^\alpha}+\frac{C_\alpha}{\EE}\mathop{{\fint}}_\mathbb{T}\frac{\phi'(\tau)d\tau}{\vert \EE\phi(w)+\EE \phi(\tau)-2d\vert^\alpha}\\
&\equiv&I_1(\EE,f(w))+I_2(\EE,f(w)).
\end{eqnarray}
We shall split,  as before,  $G^\alpha$ into three terms
\begin{equation}\label{split1X}
G^\alpha=G_1+G_2+G_3
\end{equation}
with
$$
G_1(\EE,\Omega,f(w))=-U\,\textnormal{Im}\Big\{ \overline{w}\,\Big(1+\EE^{1+\alpha}\overline{f^\prime(w)}\Big)\Big\},
$$
$$
G_2(\EE,f(w))=\textnormal{Im}\Big\{I_1(\EE,f(w))\, \overline{w}\,\Big(1+\EE^{1+\alpha}\overline{f^\prime(w)}\Big)\Big\}
$$
and 
$$
G_3(\EE,f(w))=\textnormal{Im}\Big\{I_2(\EE,f(w))\, \overline{w}\,\Big(1+\EE^{1+\alpha}\overline{f^\prime(w)}\Big)\Big\}.
$$
\section{Existence of corotating vortex pairs}\label{co-E}

In this section we will prove the existence of rotating pairs of patches for the $(\hbox{SQG})_\alpha$ model
with $\alpha\in[0,1)$.  Recall that the equations governing the boundaries of the vortices were formulated
in the subsection \ref{SEc-rotat11}.
The first goal  is to discuss the regularity of the functionals defining the V-states and to compute 
the associated linear operator at the the trivial solutions corresponding to the point vortex pair.
 In the subsection \ref{real-velo} we shall see how the angular velocity is uniquely determined through 
 the geometry of the domain. In this setting $\Omega$ plays the role of the Lagrangian multiplier such 
 that the first Fourier coefficient of the nonlinear functional vanishes. Finally,  in the \mbox{subsection \ref{Secw1},}  we shall see  that  
 the existence of the vortex pairs is a simple  consequence of  
the implicit function theorem in suitable  Banach spaces and  the convexity of each 
single patch is done in a standard way through  a perturbative argument applied for the curvature.

\subsection{Extension and regularity of the functional $G^\alpha$}

The main idea to prove the existence of rotating vortex pairs is to apply the implicit function 
theorem to the equations { \eqref{def1} and \eqref{model}.}
To this end we have to check that the functions $G^\alpha$ defined in \eqref{def1} and \eqref{model} 
satisfy some regularity conditions. The spaces which are relevant in this study and used  throughout 
the paper are  described below. 
Take $\alpha\in [0,1)$ and $\beta\in(0,1)$ an arbitrary given number, then  we define the spaces

\begin{equation*}
X^\alpha\equiv\left\{ \begin{array}{ll}
\Big\{ f\in C^{1+\beta}(\mathbb{T}), f(w)=\sum_{n\geq1} a_n w^{-n}, a_n\in\RR\Big\},\quad\hbox{if}\quad \alpha=0\\
\Big\{ f\in C^{2-\alpha}(\mathbb{T}), f(w)=\sum_{n\geq1} a_n w^{-n}, a_n\in\RR\Big\},\quad\hbox{if}\quad \alpha\in (0,1)
\end{array} \right.
\end{equation*}
and
\begin{equation*}
Y^\alpha\equiv\left\{ \begin{array}{ll}
\Big\{ f\in C^{\beta}(\mathbb{T}), f(w)=\sum_{n\geq1} a_n e_n, a_n\in\RR\Big\},\quad\hbox{if}\quad \alpha=0\\
\Big\{ f\in C^{1-\alpha}(\mathbb{T}), f(w)=\sum_{n\geq1} a_n e_n, a_n\in\RR\Big\},\quad\hbox{if}\quad \alpha\in (0,1)
\end{array} \right.
\end{equation*}
with the notation $e_n(w)=\hbox{Im}(w^n)$. We shall also consider the subspaces

$$\widehat{Y}^\alpha=\big\{f\in Y^\alpha, a_1=0\big\}.
$$
For $r>0$ we denote by  $B_r^\alpha$ the open ball of $X^\alpha$ centered at zero and of radius $r.$ The next result deals with  some properties of the function $G^\alpha.$ Before giving the main statement  of this section it is convenient to introduce the notation
\begin{equation*}
\widehat{C}_\alpha\equiv
\alpha C_\alpha=2^\alpha\frac{ \Gamma(\frac{2+\alpha}{2})}{\Gamma(\frac{2-\alpha}{2})},\quad\hbox{for}\quad \alpha\in [0,1).
\end{equation*}


\begin{proposition}\label{prop1}
Let $\alpha\in[0,1),$ then the function $G^\alpha$ can be extended from
$(-\frac12,\frac12)\times\RR\times B_1^\alpha$ to $Y^\alpha$ as  a $C^1$ function. 
Moreover, for any $\Omega\in \RR$  the operator 
$ \partial_fG^\alpha(0,\Omega,0):X^\alpha\to \widehat Y^\alpha$ is an isomorphism.
More precisely, for   $\displaystyle{h=\sum_{n\geq1} a_n {w^{-n}}\in X^\alpha}$, we get
 $$
 \partial_fG^\alpha(0,\Omega,0)h(w)=\sum_{n\geq1} a_n \widehat\gamma_n e_{n+1} 
 $$
  with
 \begin{eqnarray*}
\widehat \gamma_n&=&\frac{\widehat{C}_\alpha\Gamma(1-\alpha)}{4\Gamma^2(1-\frac\alpha2)}\bigg( \frac{2(1+n)}{1-\frac\alpha2}-\frac{\big(1+\frac{\alpha}{2}\big)_{n}}{\big(1-\frac{\alpha}{2}\big)_{n}}-\frac{\big(1+\frac{\alpha}{2}\big)_{n+1}}{\big(1-\frac{\alpha}{2}\big)_{n+1}}\bigg).
\end{eqnarray*}

\end{proposition} 
\begin{remark}
We can easily check from the proofs that two initial point vortex ${\pi \delta_{(0,0)}}$ and $\pi \delta_{(2d,0)}$ rotate uniformly about $(d,0)$ with the angular velocity
 \begin{equation}\label{Osing}
\Omega_{sing}^\alpha\equiv \frac{\widehat{C}_\alpha}{(2d)^{2+\alpha}}\cdot
 \end{equation}
\end{remark}
\begin{remark}\label{Reg-up}
By adapting the proof below we can check that the preceding proposition remains true if we change in the definition of $X^\alpha$ and $Y^\alpha$ the parameter $\alpha$ by $\alpha-n$ for any $n\in\NN^\star.$ As a consequence, the boundaries of the V-states belong to  the H\"{o}lderian class $C^{n-\alpha}$ for \mbox{any $n\in\NN.$}
\end{remark}
\begin{proof}
The proof will be divided in two pieces. In the first one we shall discuss the case $\alpha=0$ and the second one will be devoted to $\alpha\in (0,1)$.
\\
{\bf{Part $\hbox{I}$: Euler case}}.  According to the definition \eqref{def1} we have the decomposition
\begin{eqnarray*}
-2G^0&\equiv& G_1+G_2+G_3\\
&=&\hbox{Im}(F_1)+\hbox{Im}(F_2)+\hbox{Im}(F_3).
\end{eqnarray*}
We will first proceed with the regularity of the  first term, that is, 
$$
G_1(\EE, \Omega, f)=\textnormal{Im}\Big\{2\Omega \Big(\EE \overline{w}+\EE^2\overline{f(w)}-d\Big) w\big(1+\EE f^\prime(w)\big)-f^\prime(w)\Big\}. 
$$
Clearly this function can be defined from the set $(-\frac{1}{2},\frac{1}{2})\times \mathbb{R}\times B_1^0$ to $Y^0$ because the function in the brackets is in $C^{\beta}(\mathbb T),$ and is obtained as sums and products of functions with real coefficients.
In order to prove its differentiability  we have to compute their partial derivatives,
$$
\partial_{\varepsilon}G_1(\EE, \Omega, f)= \textnormal{Im}\Big\{2\Omega( \overline{w}+2\varepsilon\overline{f(w)})w(1+\varepsilon f'(w))+2\Omega(\varepsilon\overline{w}+\varepsilon^2\overline{f(w)}-d)wf'(w)\Big\},
$$
and clearly this is a continuous function from $(-\frac{1}{2},\frac{1}{2})\times \mathbb{R}\times B_1^0$ to $Y^0.$ 
\\
Taking now the derivative in $\Omega$ we get
$$
\partial_{\Omega}G_{1}(\EE, \Omega, f)= \textnormal{Im}\Big\{2\Big(\EE \overline{w}+\EE^2\overline{f(w)}-d\Big) w\big(1+\EE f^\prime(w)\big)\Big\},
$$
 which is continuous from $(-\frac{1}{2},\frac{1}{2})\times \mathbb{R}\times B_1^0$ to $Y^0.$ 
Let us note that $G_1$ is a polynomial also in $f$ and $f'$ and consequently the derivative is  polynomial in $f$ and $f'$. 
Thus, it is necessary a continuous function  from $(-\frac{1}{2},\frac{1}{2})\times \mathbb{R}\times B_1^0$ to $Y^0.$  It is an easy computation 
to check that 
$$
\partial_{f}G_1(0,\Omega, 0)(h)=-\textnormal{Im}\{h'(w)\}. 
$$
Let's take now 
\begin{eqnarray*}
G_2(\EE,f)&=&\textnormal{Im}\Big\{  \Big(\fint_{\mathbb{T}}\frac{\overline{A}+\EE\overline{B}}{A+\EE B} f^\prime(\tau)d\tau+\fint_{\mathbb{T}}\frac{A\overline{B}-\overline{A} B}{A(A+\EE B)}d\tau\Big)w\big(1+\EE f^\prime(w)\big)  \Big\} \\
&=&\textnormal{Im}\Big\{ \big(G_{21}+G_{22}\big)w\big(1+\EE f^\prime(w)\big) \Big\}.
\end{eqnarray*}
To prove that $G_2(\EE,f)$ is a function from $(-\frac{1}{2},\frac{1}{2})\times \mathbb{R}\times B_1^0$ to $Y^0$ it is enough 
to verify that the functions $G_{21}(\EE,f)$ and $G_{22}(\EE,f)$ satisfy the same property. 
Observe that the function 
$$
G_{21}(\EE,f)=\fint_{\mathbb T}\frac{\overline{\tau}-\overline{w}+\EE(f(\overline{\tau})-f(\overline{w}))}{\tau-w+\EE(f(\tau)-f(w))}f^\prime(\tau)d\tau
$$
is given by an integral operator. Since $f$ is in $C^{1+\beta}(\mathbb T),$ we will have that $G_{21}$ belongs to  the space $C^{\beta}(\mathbb T)$ provided  the kernel
$$ 
K(\tau,w)=\frac{\overline{\tau}-\overline{w}+\EE(f(\overline{\tau})-f(\overline{w}))}{\tau-w+\EE(f(\tau)-f(w))}
$$
satisfies the hypotheses of Lemma  \ref{noyau} for $\alpha=0.$
It is obvious that
$$
\sup_{\tau\neq w}|K(\tau,w)|\le 1, 
$$
and moreover 
\begin{eqnarray*}
|\partial_wK(\tau,w)|&=&\Big|\frac{(1+\EE f^\prime(w)\big((\overline{\tau}-\overline{w})+\EE(f(\overline{\tau})-f(\overline{w})\big)}
{((\tau-w)+\EE(f(\tau)-f(w)))^2}      
+\frac{1}{w^2}\frac{1+\EE f^\prime(\overline{w})}{(\tau-w)+\EE(f(\tau)-f(w))}
\Big|\\
&\le&\frac{M^2+M}{|\tau-w|},
\end{eqnarray*}
where $M=\frac{1+\EE\|f\|_{C^{1+\alpha}(\mathbb T)}}{1-\EE\|f\|_{C^{1+\alpha}(\mathbb T)}}.$ 
Now to check that this function has real coefficients we have to show that
$\overline{G_{21}(\EE,f)(w)}=G_{21}(\EE,f)(\overline w).$
Using the change of variable $\eta=\overline\tau,$ it is an easy computation to see that 
\begin{eqnarray*}
\overline{G_{21}(\EE,f)(w)}&=-&\fint_{\mathbb T}\frac{\tau-w+\EE(f(\tau)-f(w))}{\overline{\tau}-\overline{w}+\EE(f(\overline{\tau})-f(\overline{w}))}f^\prime(\overline{\tau})d\overline{\tau}=
\fint_{\mathbb T}\frac{\overline{\eta-\overline w}+\EE\overline{(f(\eta)-f(\overline w))}}{\eta-\overline w+\EE(f(\eta)-f(\overline w))}f^\prime({\eta})d{\eta}\\
&=&G_{21}(\EE,f)(\overline w).
\end{eqnarray*}
On the other hand the function 
$$
G_{22}(\EE,f)=\fint_{\mathbb T} \frac{(\tau-w)\big(f(\overline{\tau})-f(\overline{w})\big)-(\overline{\tau}-\overline{w}) (f(\tau)-f(w))}
{(\tau-w)\big((\tau-w)+\EE(f(\tau)-f(w)\big)}d\tau
$$
will be in the space $C^{\beta}(\mathbb T)$ if the kernel 
$$
K(\tau,w)=\frac{(\tau-w)\big(f(\overline{\tau})-f(\overline{w})\big)-(\overline{\tau}-\overline{w}) (f(\tau)-f(w))}
{(\tau-w)\big((\tau-w)+\EE(f(\tau)-f(w)\big)}
$$
satisfies the hypotheses of Lemma \ref{noyau} for $\alpha=0$. As before, it is straightforward that
$$
\sup_{\tau\neq w}|K(\tau,w)| \le \frac{2\| f \| _{C^{1+\alpha}}}{1-\EE\| f\| _{{C^{1+\alpha}(\mathbb T)}}}
$$
and  
$$
|\partial_wK(\tau,w)\le \frac{C}{|\tau-w|},
$$
where the constant $C$ depends on $\EE$ and $\|f\|_{C^{1+\beta}(\mathbb T)}.$ 
To check that the function $G_{22}$ has real coefficients one can repeat the same procedure 
used before   for the function $G_{21}.$ 
\\
Now we will verify that the function $G_2$ is of class  $C^1$  from $(-\frac{1}{2},\frac{1}{2})\times \mathbb{R}\times B_1^0$ to $Y^0.$ 
To do so,  we will check the continuity of the partial derivatives of  $G_{21}$ and $G_{22}.$
Simple computations prove that 
\begin{eqnarray*}
\partial_{\EE}G_{21}&=&\fint_{\mathbb T}\frac{f(\overline{\tau})-f(\overline{w})}{\tau-w+\EE(f(\tau)-f(w))}f^\prime(\tau)d\tau\\
&-&\fint_{\mathbb T}
\frac{\overline{\tau}-\overline{w}+\EE\big(f(\overline{\tau})-f(\overline{w})\big)}{(\tau-w+\EE(f(\tau)-f(w))^2}(f(\tau)-f(w)) f^\prime(\tau)d\tau
\end{eqnarray*}
and 
$$
\partial_{\EE}G_{22}=-2i\fint_{\mathbb T}\frac{\textnormal{Im}\Big\{(\tau-w)\big(f(\overline{\tau})-f(\overline{w})\big)\Big\}}{(\tau-w)\big(\tau-w+\EE(f(\tau)
-f(w))\big)^2}(f(\tau)-f(w))d\tau.
$$
The existence and the continuity of this partial derivative can be obtained by proving that the kernels 
that appear in the integral operators 
satisfy the conditions of Lemma  \ref{noyau}.
\mbox{Take $h \in X$} we will compute the G\^ateaux derivative in the direction $h$ of the function $G_2.$
 For it we only need to calculate the G\^ateaux derivatives of the functions $G_{21}$ and $G_{22}.$
\begin{eqnarray*}
 \partial_fG_{21}(\EE,f)h(w)&\equiv&\fint_{\mathbb T}\frac{(h(\overline{\tau})-h(\overline{w}))}
 {\tau-w+\EE(f(\tau)-f(w))}f^\prime(\tau)d\tau
\\
&+&\fint_{\mathbb T}\frac{\overline{\tau}-\overline{w}+\EE(f(\overline{\tau})-f(\overline{w}))}{\tau-w+\EE(f(\tau)-f(w))}h^\prime(\tau)d\tau\\
&-&\EE\fint _{\mathbb T}\frac{\overline{\tau}-\overline{w}+\EE(f(\overline{\tau})-f(\overline{w}))}{\big(\tau-w+\EE(f(\tau)-f(w)))^2} (h(\tau)-h(w)
\big)f^\prime(\tau)d\tau.
 \end{eqnarray*}
 Moreover, one can easily check that
 \begin{eqnarray*}
 \partial_fG_{22}(\EE, f)h(w)&=&2i\fint_{\mathbb T}\frac{\textnormal{Im}\big\{(\tau-w)(h(\overline{\tau})-h(\overline{w}))\big\}}{(\tau-w)(\tau-w+\EE(f(\tau)-f(w))}d\tau\\
 &-&2i\EE\fint_{\mathbb T}\frac{\textnormal{Im}\big\{(\tau-w)(f(\overline{\tau})-f(\overline{w}))\big\}}{(\tau-w)\big((\tau-w+\EE(f(\tau)-f(w))\big)^2}(h(\tau)-h(w))d\tau.
 \end{eqnarray*}
Again  Lemma \ref{noyau} applied to the kernels that appear in the G\^ateaux derivatives of the functions $G_{21}$ and $G_{22}$ 
will give  the existence and the continuity of the functions $\partial_fG_{21}$ and $\partial_fG_{22}.$
On the other hand,
$$
\partial_{f}G_2(0,0)(h)=\textnormal{Im}\Big\{\Big(\partial_fG_{21}(0,0)(h)-\partial_fG_{22}(0,0)(h)\Big)w\Big\}.
$$
In addition, using  the residue theorem we get 
$$
\partial_{f} G_{21}(0,0)(h)=\fint_{\mathbb T}\frac{\overline{\tau}-\overline{w}}{\tau-w}h^\prime(\tau)d\tau=0
$$
and 
$$
\partial_{f} G_{22}(0,0)(h)=2i\fint_{\mathbb T}\frac{\textnormal{Im}\{(\tau-w)(h(\overline{\tau})-h(\overline{w}))\}}{(\tau-w)^2}d\tau=0.
$$
Consequently $\partial_{f}G_2(0,0)(h)=0.$
Let's now study the last  function  in \eqref{def1}
\begin{eqnarray*}
G_{3}(\EE,f)&=&-\textnormal{Im}\Big\{\Big(\fint_{\mathbb T}\frac{\overline\tau+\EE f(\overline{\tau})}{\EE(\tau+w)+\EE^2(f(\tau)+f(w))-2d}(1+\EE f^\prime(\tau))d\tau\Big)w(1+\EE f^\prime(w))\Big\}\\
&=&-\textnormal{Im}\Big\{G_{31}(\EE,f)w(1+\EE f^\prime(w)\Big\}.
\end{eqnarray*}
So, the regularity of the function $G_{3}$ is equivalent to the regularity of the function $G_{31}.$ 
Now this function is given by an integral operator  with  kernel 
$$
K(\tau,w)=\frac{\overline\tau+\EE f(\overline{\tau})}{\EE(\tau+w)+\EE^2(f(\tau)+f(w))-2d}\cdot
$$
It is clear that  $|K(\tau,w)|\le C$ and moreover
$$
|\partial_{w} K(\tau,w)|=\Big|\frac{(\overline\tau+\EE f(\overline{\tau}))\big(\EE+\EE^2f^\prime(w)\big)}{\big(\EE(\tau+w)+
\EE^2(f(\tau)+f(w))-2d\big)^2}\Big|\le C.
$$
Since $1+\EE f^\prime(\tau)$ is in $C^{\beta}(\mathbb T)$ then  applying once again  Lemma \ref{noyau} 
to the above kernel we get that $G_{31}$ is a function in $C^{\beta}(\mathbb T)$. 
To prove that $G_{31}$ has real coefficients one only has to repeat the arguments given 
in the case of the function $G_{21}.$
Now, to check that the function $(\EE,f)\mapsto G_{31}(\EE, f)$ is  $C^1$ we have to compute 
its  partial derivatives
\begin{eqnarray*}
\partial_{\EE}G_{31}&=&\fint_{\mathbb T}\frac{f(\overline\tau)(1+\EE f^\prime(\tau))}{\EE (\tau +w)+\EE ^2(f(\tau)+f(w))-2d}d\tau
+\fint_{\mathbb T}\frac{(\overline\tau+\EE f(\overline\tau))f^\prime(\tau)}{\EE (\tau +w)+\EE ^2(f(\tau)+f(w))-2d}d\tau\\
&-&\fint_{\mathbb T}\frac{(\overline\tau+\EE f(\overline\tau)) (\tau+w+2\EE(f(\tau)+f(w))}{\big(\EE (\tau +w)+\EE ^2(f(\tau)+f(w))-2d\big)^2}(1+\EE f^\prime(\tau))d\tau.
\end{eqnarray*}
Easy computations, using Lemma \ref{noyau},  prove that these operators are continuous  
from $(-\frac{1}{2},\frac{1}{2})\times \mathbb{R}\times B_1^0$ to $C^{\beta}(\mathbb T).$ 
Since they are  functions with real coefficients we can conclude that $\partial_{\EE}G_3$ is continuous 
from $(-\frac{1}{2},\frac{1}{2})\times \mathbb{R}\times B_1^0$ to $Y^0.$
On the other hand, we can compute the G\^ateaux derivative of $G_{31}$ in a given  direction $h\in X$
\begin{eqnarray*}
\partial_fG_{31}(\EE, f)(h)&=&\EE \fint_{\mathbb T}\frac{h(\overline\tau)(1+\EE f^\prime(\tau))}{ \EE (\tau +w)+\EE ^2(f(\tau)+f(w))-d}d\tau\\
&+&\EE \fint_{\mathbb T}\frac{(\overline\tau+\EE f(\overline\tau))h^\prime(\tau)}{\EE (\tau +w)+\EE ^2(f(\tau)+f(w))-d}d\tau\\
&-&\EE^2 \fint_{\mathbb T}\frac{(\overline\tau+\EE f(\overline\tau))(h(\tau)+h(w))}{\big(\EE (\tau +w)+\EE ^2(f(\tau)+f(w))-d\big)^2}(1+\EE f^\prime(\tau))d\tau.
\end{eqnarray*}
Again, by straightforward  computations one can verify   that the integral operators defined by
these partial derivatives are continuous 
 and so we obtain that $\partial _{f}G_3$ is continuous from 
 $(-\frac{1}{2},\frac{1}{2})\times \mathbb{R}\times B_1^0$ to $Y^0.$ 
Moreover we find $
\partial_{f}G_{31}(0,0)(h)=0,
$
and consequently 
$$
{\partial_{f}G_{31}(0,0)(h)=0.}
$$
Since by definition {$-2G^0=G_1+G_2+G_3$} then  we deduce  that  $\partial_fG^0(0,\Omega,0)(h)=\frac12\textnormal{Im}\{h^\prime\},$ which is clearly an isomorphism  from $X^0$ to $ \widehat Y^0$. Finally we note that 
when $\EE=0$ one should get the two point vortices. Indeed, we can easily check that
$$
G^0(0,\Omega,0)=\textnormal{Im}\Big\{\big(\Omega d\, -\frac{1}{4d}\big)w\Big\}
$$
and therefore $G^0(0,\Omega,0)=0$ if and only if
\begin{equation*}
\Omega=\Omega_{sing}^0=\frac{1}{4 d^2}\cdot
\end{equation*}
\\
{\bf{Part $\hbox{II}$: case $\alpha\in (0,1)$}}. Now we shall move to the proof of the statement when $\alpha\in (0,1).$ Recall that the functional defining the rotating pairs is given in \eqref{split1}. 
We shall start with the proof of the regularity for the function $G_1$ which is the easiest one.
 The suitable extension of this function, still denoted $G_1$,  is given by
\begin{equation}\label{G_0X1}
{G_1}(\EE,\Omega,f(w))= \Omega\,\textnormal{Im}\Bigg\{\Big(\EE w+\EE ^2|\EE|^{\alpha}f(w)-d\Big)\, \overline{w}\,
\Big(1+\EE|\EE|^\alpha\overline{f^\prime(w)}\Big)\Bigg\}.
\end{equation}
 $G_1$ is well-defined from $(-\frac12,\frac12)\times\RR\times B_1^\alpha$ to $ Y^\alpha$ because, using the algebra structure of H\"older 
spaces with positive regularities,  the function between the brackets
 belongs to $C^{1-\alpha}(\mathbb{T})$ and has real Fourier coefficients.
To prove that this functional is $C^1$ it suffices to check that the partial derivatives exist and are continuous. It is clear that
\begin{eqnarray*}
\partial_\EE{G_1}(\EE,\Omega,f(w))&=&\Omega\,\textnormal{Im}\Bigg\{\Big(w+(2\EE|\EE|^\alpha+\alpha\, \hbox{sign}(\EE) |\EE|^{1+\alpha}) f(w)\Big)\, \overline{w}\,\Big(1+\EE|\EE|^\alpha\overline{f^\prime(w)}\Big)\Bigg\}\\
&+&\Omega(|\EE|^{\alpha}+\alpha\, \hbox{sign}(\EE)\EE|\EE|^{\alpha-1})\textnormal{Im}\Bigg\{\Big(\EE w+\EE ^2|\EE|^{\alpha}f(w)-d\Big)\, \overline{w}\,\overline{f^\prime(w)}\Bigg\}
\end{eqnarray*}
This function is polynomial in $f$ and $f^\prime$ and therefore it is  continuous from $(-\frac12,\frac12)\times\RR\times B_1^\alpha$ to $ Y^\alpha$.  
On the other hand, the partial derivative with respect to $\Omega$ it  is given by
$$
\partial_\Omega{G_1}(\EE,\Omega,f(w))=\textnormal{Im}\Bigg\{\Big(\EE w+\EE ^2|\EE|^{\alpha}f(w)-d\Big)\, \overline{w}\,
\Big(1+\EE|\EE|^\alpha\overline{f^\prime(w)}\Big)\Bigg\}
$$
and this  is obviously  continuous from  $(-\frac12,\frac12)\times\RR\times B_1^\alpha$ to $ Y^\alpha$.  
Note also that  ${G_1}$ is polynomial with respect to $f$ and $f^\prime$ and consequently $\partial_f {G_1}$ exists and is continuous. 
This concludes the fact that ${G_1}$ is $C^1.$ It is  easy to check that for any direction $h\in X^\alpha$
\begin{equation}\label{F1}
\partial_f {G_1}(0,\Omega,0)(h)=0.
\end{equation}
For the remaining functionals the situation is much more complicated. As we shall see the reasoning is very classical and we will  
give just  some significant  details. We first start with the term $G_3$. To find the suitable extension note that  the ansatz 
of the solution is very crucial and allows to get rid of the singularity in $\EE.$ Recall that
\begin{equation}\label{G_2X1}
G_3(\EE,f(w))=\textnormal{Im}\Big\{I_2(\EE,f(w)) L(\EE,f(w))\Big\}\quad\hbox{with}\quad  L(\EE,f(w))=\overline{w}\,
\Big(1+\EE|\EE|^\alpha\overline{f^\prime(w)}\Big)
\end{equation} 
where we have extended  the tangent vector to $L$ and 
as previously,  
$$L:(-\frac12,\frac12)\times B_1^\alpha\to  Y^\alpha
$$ is well-defined and is of class $C^1$. Therefore it suffices to prove that $I_2$ can be extended from 
$(-\frac12,\frac12)\times \RR\times B_1^\alpha$ to $Y^\alpha$ as a $C^1$ function. The key point is Taylor formula:  
\begin{equation}\label{Ta1}
\frac{1}{|A+B|^\alpha}=\frac{1}{|A|^\alpha}-\alpha\int_0^1\frac{\hbox{Re}(A\overline{B})+t|B|^2}{|A+tB|^{2+\alpha}}dt
\end{equation}
which is true for any complex numbers $A,B$ such that $B|<|A|$. As an application we get
$$
\frac{1}{\vert \EE\phi(w)+\EE \phi(\tau)-2d\vert^\alpha}=\frac{1}{(2d)^\alpha}-\alpha
\int_0^1\frac{-2d\,\EE\hbox{Re}\big[\phi(\overline{\tau})+\phi(\overline{w})\big]+t\EE^2|\phi(\tau)+\phi(w)|^2}
{\vert t\EE\phi(w)+t\EE \phi(\tau)-2d\vert^{2+\alpha}}dt.
$$
We mention that the condition $|B|<|A|$ is satisfied because
\begin{eqnarray*}
\big|\EE\phi(w)+\EE \phi(\tau)\big|&\le& 2\EE\|\phi\|_{L^\infty}\\
&\le&4\EE\\
&<& d.
\end{eqnarray*}
Consequently
\begin{equation*}
{I_2}(\EE,f(w))=-\alpha C_\alpha\fint_{\mathbb{T}}\int_0^1\frac{-2d\,\hbox{Re}
\big[\phi(\overline{\tau})+\phi(\overline{w})\big]+t\EE|\phi(\tau)+\phi(w)|^2}{\vert t\EE\phi(w)+t\EE \phi(\tau)-2d
\vert^{2+\alpha}}\phi^\prime(\tau) d\tau dt,
\end{equation*}
where we have used the fact
$$
\fint_{\mathbb{T}}\phi^\prime(\tau)d\tau=0.
$$
Thus the suitable extension of this functional is 
\begin{eqnarray}\label{I_2X}
{I_2}(\EE,f(w))&=&-\alpha C_\alpha\fint_{\mathbb{T}}\int_0^1\frac{-2d\,\hbox{Re}\big[\phi(\overline{\tau})
+\phi(\overline{w})\big]+t\EE|\phi(\tau)+\phi(w)|^2}{\vert t\EE\phi(w)+t\EE \phi(\tau)-2d\vert^{2+\alpha}}\phi^\prime(\tau) d\tau dt\\
\nonumber &\equiv&-\alpha C_\alpha\fint_{\mathbb{T}}\int_0^1\mathcal{K}_2(\tau,w)\phi^\prime(\tau)d\tau,
\end{eqnarray}
with $\phi(w)=w+\EE|\EE|^\alpha f(w)$. We shall now check that this extension defines a $C^1$ function  
 from   $(-\frac12,\frac12)\times B_1^\alpha$ to $  Y^\alpha$. First, the integral operator is well-defined since the kernel
  $\mathcal{K}_2$ is not singular 
and satisfies the hypotheses of  Lemma \ref{noyau} for any $f\in B_1^\alpha,$
$$
|\mathcal{K}_2(\tau,w)|\le C\quad \hbox{and}\quad |\partial_w\mathcal{K}_2(\tau,w)|\le C
$$
  for some constant $C$ and thus
  $$
  \|I_2(\EE,f)\|_{C^{1-\alpha}(\mathbb{T})}\le C\|\phi^\prime\|_{L^\infty}\le C.
  $$
Taking $(\EE,f)=(0,0)$ in \eqref{I_2X} yields
\begin{eqnarray}\label{I_2XXX}
\nonumber I_2(0,0)&=&\frac{\alpha C_\alpha}{(2d)^{1+\alpha}}\fint_{\mathbb{T}}\int_0^1\overline{\tau} d\tau dt\\
&=&\frac{\alpha C_\alpha}{2(2d)^{1+\alpha}}.
\end{eqnarray}
In addition, for  $f\in B_1^\alpha$ and $h\in X^\alpha$ one has 
\begin{eqnarray*}
\partial_f I_2(0,f)h(w)&=&\frac{d}{ds}I_2(0, f(w)+s\,h(w))_{|s=0}\\
&=&0.
\end{eqnarray*}
As 
$
\partial_f L(0,f)=0
$
then we deduce 
\begin{equation}\label{F_2}
\partial_f G_3(0,f)=0.
\end{equation}
For a future use, we shall apply once again \eqref{Ta1} to $I_2(\EE,f)$ in order to get
\begin{eqnarray}\label{I_2XX}
\nonumber {I_2}(\EE,f(w))&=&\frac{\alpha C_\alpha}{(2d)^{1+\alpha}}\fint_{\mathbb{T}}\hbox{Re}
\big(\phi(\overline{\tau})\big)\phi^\prime(\tau)d\tau -\frac{\alpha C_\alpha}{(2d)^{2+\alpha}}\frac{\EE}{2}\fint_{\mathbb{T}}
{|\phi(\tau)+\phi(w)|^2}\phi^\prime(\tau) d\tau 
\\
&+&\alpha C_\alpha(2+\alpha)\EE\fint_{\mathbb{T}}{K_2(\tau,w)}\phi^\prime(\tau)d\tau,
\end{eqnarray}
$$
 K_2(\tau,w)=\iint_{[0,1]^2}\frac{\Big(-2d\,\hbox{Re } \Phi({\tau},w)+t\EE|\Phi({\tau},w)|^2\Big)\Big(-2d\,
 t\,\hbox{Re } \Phi(\tau,w)+st\EE|\Phi(\tau,w)|^2\Big)}{\vert st\EE\phi(w)
+st\EE \phi(\tau)-2d\vert^{4+\alpha}}dtds,
$$
with
$$
\Phi(\tau,w)=\phi(\tau)+\phi(w).
$$
On the other hand,  since $\phi(w)=w+\EE|\EE|^\alpha f(w)$ 
\begin{eqnarray*}
\nonumber \fint_{\mathbb{T}}\hbox{Re}\big(\phi(\overline{\tau})\big)\phi^\prime(\tau)d\tau
&=&\frac12+\EE|\EE|^\alpha\fint_{\mathbb{T}} \hbox{Re}(\tau)f^\prime(\tau)d\tau+\EE|\EE|^\alpha\fint_{\mathbb{T}} 
\hbox{Re}\big(f(\tau)\big)\phi^\prime(\tau)d\tau.
\end{eqnarray*}
Therefore one gets
\begin{equation}\label{FF-11}
{I_2}(\EE,f(w))=\frac{\alpha C_\alpha}{2(2d)^{1+\alpha}}+\EE \,\mathcal{I}_2(\EE,f(w)).
\end{equation}
Using that the kernel $K_2$ satisfies the conditions of Lemma \ref{noyau}, that is,
$$|K_2(\tau,w)|\le C,\quad |\partial_w K_2(\tau,w)|\le C,
$$
one can verify that the function  ${I}_2:(-\frac12,\frac12)\times B_1^\alpha \to C^{1-\alpha}(\mathbb{T})$ is well-defined.
To prove that it is indeed of class $C^1$ we shall look for its derivatives and study their continuity.  The computations are straightforward and resemble to those done for Euler equations and thus we will skip the details.  
Inserting the formula \eqref{FF-11} into the expression of $G_3$ allows to get the decomposition
\begin{equation}\label{G_2}
G_3(\EE,f)\equiv-\frac{\alpha C_\alpha}{2(2d)^{1+\alpha}} e_1(w)-\EE \mathcal{R}_2(\EE,f)
\end{equation}
with $\mathcal{R}_2:(-\frac12,\frac12)\times B_1^\alpha\to C^{1-\alpha}(\mathbb{T})$ being a  $C^1$ function.
Let us now move to the extension  of the function $G_2$ defined in \eqref{split1X}. We can split an extension of $I_1(\EE,f)$ into  three parts as follows,
\begin{eqnarray}\label{FF1}
\nonumber I_1(\EE,f(w))&=&{C_\alpha}\mathop{{\fint}}_\mathbb{T}\frac{f^\prime(\tau)d\tau}{\vert \phi(w)-\phi(\tau)\vert^\alpha}+\frac{C_\alpha}{\EE |\EE|^{\alpha}}{{\fint}}_\mathbb{T}\frac{d\tau}{\vert \tau-w\vert^\alpha}\\
\nonumber &+&\frac{C_\alpha}{\EE |\EE|^{\alpha}}\mathop{{\fint}}_\mathbb{T}\Big(\frac{1}{\vert \phi(w)-\phi(\tau)\vert^\alpha}-\frac{1}{\vert \tau-w\vert^\alpha}\Big)d\tau\\
\nonumber &=&{C_\alpha}\mathop{{\fint}}_\mathbb{T}\frac{f^\prime(\tau)d\tau}{\vert \phi(w)-\phi(\tau)\vert^\alpha}+\frac{\widehat \mu_\alpha\, w}{\EE |\EE|^{\alpha}}+{C_\alpha}\mathop{{\fint}}_\mathbb{T}K(\EE,\tau,w) d\tau\\
&=&I_{11}(\EE,f(w))+\frac{\widehat \mu_\alpha\, w}{\EE |\EE^{\alpha}|}+I_{12}(\EE,f(w))
\end{eqnarray}
with
$$
K(\EE,\tau,w)=\frac{1}{\EE |\EE|^{\alpha}}\Big(\frac{1}{\vert \phi(w)-\phi(\tau)\vert^\alpha}-\frac{1}{\vert \tau-w\vert^\alpha}\Big)
$$
and
\begin{equation}\label{C_XX}
\widehat \mu_\alpha=\frac{\alpha \Gamma(1-\alpha)}{(2-\alpha)\Gamma^2(1-\frac\alpha2)} C_\alpha=\frac{ \Gamma(1-\alpha)\Gamma(1+\frac\alpha2)}{2^{1-\alpha}\Gamma(2-\frac\alpha2)\Gamma^2(1-\frac\alpha2)}\cdot
\end{equation}
Note that we  have used the identity, see \cite[Lemma 2]{H-H}
$$
\fint_{\mathbb{T}}\frac{d\tau}{|\tau-w|^\alpha}=\frac{\alpha \Gamma(1-\alpha)}{(2-\alpha)\Gamma^2(1-\frac\alpha2)} w.
$$
Consequently,
\begin{eqnarray}\label{F001}
G_2(\EE,f(w))
\nonumber &=&\hbox{Im}\Big\{I_{11}(\EE,f(w))\overline{w}\Big(1+\EE|\EE|^\alpha\overline{f^\prime(w)}\Big)\Big\}\\
&+&\hbox{Im}\Big\{I_{12}(\EE,f(w)) \overline{w}\Big(1+\EE|\EE|^\alpha\overline{f^\prime(w)}\Big)\Big\}-\widehat \mu_\alpha \hbox{Im}(f^\prime(w)).
\end{eqnarray}
The last term defines a linear operator from $X^\alpha$ to $Y^\alpha$ and therefore it is smooth. It remains to study  the first and second terms. This amounts to studying the terms $I_{11}$ and $I_{12}$. The first part is extended as usual through the formula
$$
I_{11}(\EE,f(w))={C_\alpha}\mathop{{\fint}}_\mathbb{T}\frac{f^\prime(\tau)\, d\tau}{\big\vert \phi(\tau)-\phi(w)\big\vert^\alpha},\quad \hbox{with}\quad \phi(w)=w+\EE|\EE|^\alpha f(w).
$$
First to check that $I_{11}$ is well-defined we use Corollary \ref{cor} which implies that 
\begin{eqnarray*}
\|I_{11}(\EE,f)\|_{C^{1-\alpha}(\mathbb{T})}&\le& C\|f^\prime\|_{L^\infty}\\
&\le&C\|f\|_{C^{2-\alpha}(\mathbb{T})}.
\end{eqnarray*}
Now we shall prove that $I_{11}$ is $C^1$ and for this purpose it suffices to check the existence of the partial derivatives and their continuity in strong topology. The partial derivative with respect to $\EE$ can be easily computed and we find
\begin{eqnarray*}
\partial_\EE I_{11}(\EE,f(w))&=&-{\alpha C_\alpha}(  |\EE|^\alpha +\alpha \EE\hbox{sign}(\EE)|\EE|^{\alpha-1})\mathop{{\fint}}_\mathbb{T}\frac{\hbox{Re}\big[(\overline\tau-\overline w)(f(\tau)-f(w))\big]}{\big\vert \phi(\tau)-\phi(w)\big\vert^{2+\alpha}}f^\prime(\tau)\, d\tau\\
&-&\alpha C_{\alpha}\EE |\EE|^ {\alpha}\big( |\EE|^{\alpha}+\alpha \EE \hbox{sign}(\EE )|\EE|^{\alpha-1} \big)\mathop{{\fint}}_\mathbb{T}\frac{|f(\tau)-f(w)|^2}{\big\vert \phi(\tau)-\phi(w)\big\vert^{2+\alpha}}f^\prime(\tau)\, d\tau.
\end{eqnarray*}
Introduce the kernels
$$
K_{1}(\tau,w)=\frac{\hbox{Re}\big[(\overline\tau-\overline w)(f(\tau)-f(w))\big]}{\big\vert \phi(\tau)-\phi(w)\big\vert^{2+\alpha}},\quad  K_2(\tau,w)=\frac{|f(\tau)-f(w)|^2}{\big\vert \phi(\tau)-\phi(w)\big\vert^{2+\alpha}}.
$$
Then for $\tau\neq w$
\begin{eqnarray*}
|K_1(\tau,w)|&\le& \frac{\|f\|_{Lip}}{\|\phi^{-1}\|_{Lip}^{2+\alpha}}|\tau-w|^{-\alpha}\\
&\le& C|\tau-w|^{-\alpha}.
\end{eqnarray*}
and in a similar way
\begin{eqnarray*}
|K_2(\tau,w)|
&\le& C|\tau-w|^{-\alpha}.
\end{eqnarray*}
Moreover 
\begin{eqnarray*}
|\partial_wK_1(\tau,w)|+|\partial_wK_2(\tau,w)|
&\le& C|\tau-w|^{-1-\alpha}.
\end{eqnarray*}
Therefore using Lemma \ref{noyau} we deduce that $\partial_\EE I_{11}(\EE,f)\in C^{1-\alpha}(\mathbb{T})$ and the dependence on 
$(\EE,f)\in (-\frac12,\frac12)\times B_1^\alpha$ is continuous. More details in a similar context can be found in \cite{H-H}. 
The partial derivative with respect to $f\in X^\alpha$ in the direction $h\in X^\alpha$ is given by
\begin{eqnarray*}
\partial_f  I_{11}(\EE,f)h&=&{C_{\alpha}}\mathop{{\fint}}_\mathbb{T}\frac{h^\prime(\tau)d\tau}{\big\vert \phi(\tau)-\phi(w)\big\vert^\alpha}\\
&-&\alpha {C_\alpha}\EE|\EE|^\alpha\mathop{{\fint}}_\mathbb{T}\frac{\hbox{Re}\big[(\phi(\tau)-\phi(w))(h(\overline{\tau})-h(\overline{w}))\big]f^\prime(\tau)}{\big\vert \phi(\tau)-\phi(w)\big\vert^{2+\alpha}}\, d\tau
\end{eqnarray*}
which is continuous from $ (-\frac12,\frac12)\times B_1^\alpha$ to $C^{1-\alpha}(\mathbb{T}).$
In particular we get for any $h\in X^\alpha$
\begin{equation}\label{F4}
\partial_f  I_{11}(0,0)h={C_\alpha}\mathop{{\fint}}_\mathbb{T}\frac{h^\prime(\tau)}{\big\vert \tau-w\big\vert^\alpha}\, d\tau.
\end{equation}
We shall now move to the extension and the regularity of $I_{12}$ defined in \eqref{FF1}. 
It can be extended through its kernel as follows, 
$$
K(\EE,\tau,w)=\frac{1}{\EE|\EE|^\alpha}\Big(\frac{1}{\vert \phi(w)-\phi(\tau)\vert^\alpha}-\frac{1}{\vert \tau-w\vert^\alpha}\Big),\quad \phi(w)=w+\EE|\EE|^\alpha f(w).
$$
Now using \eqref{Ta1} we find that
\begin{eqnarray}\label{diff1}
\nonumber K(\EE,\tau,w)&=&-\alpha\int_{0}^1\frac{\hbox{Re}\Big((\overline\tau-\overline{w})\big(f(\tau)-f(w)\big)\Big)}{\vert \tau-w+t\EE|\EE|^\alpha(f(\tau)-f(w))\vert^{2+\alpha}}dt\\
&-&\alpha\EE|\EE|^\alpha\int_{0}^1\frac{t\vert  f(\tau)-f(w)\vert^{2}}{\vert \tau-w+t\EE|\EE|^\alpha(f(\tau)-f(w))\vert^{2+\alpha}}dt.
\end{eqnarray}
By straightforward computations we can check that there exists an absolute constant $C$ such that for
 $(\EE,f)\in (-\frac12,\frac12)\times \in B_1^\alpha $ and  for $\tau\neq w$
$$
|K(\EE,\tau,w)|\le \frac{C}{|\tau-w|^{\alpha}}\quad\hbox{and}\quad |\partial_w K(\EE,\tau,w)|\le \frac{C}{|\tau-w|^{1+\alpha}}\cdot
$$
Therefore using Lemma \ref{noyau} again we can conclude that $I_{12}(\EE,f)$ is well-defined and belongs to 
 $C^{1-\alpha}(\mathbb{T})$. The regularity with respect to $\EE$ is straightforward since $\EE\mapsto K(\EE,\tau,w)$ is $C^1$ and 
$$
|\partial_\EE K(\EE,\tau,w)|\le \frac{C}{|\tau-w|^{\alpha}}\quad\hbox{and}\quad |\partial_w \partial_\EE K(\EE,\tau,w)|\le \frac{C}{|\tau-w|^{1+\alpha}}
$$
which implies that $\partial_\EE I_{12}(\EE,f)$ is well-defined for $(\EE,f)\in (-\frac12,\frac12)\times \in B_1^\alpha $ and it  belongs to the space $C^{1-\alpha}(\mathbb{T})$. The continuity can be done in a similar way. 
Moreover by (\ref{diff1}) we have that 
\begin{eqnarray}\label{I12}
I_{12}(0,0)=0.
\end{eqnarray}
The existence of partial derivative with respect to $f$  can be done without difficulty and we can check that this derivative is continuous. Thus we establish that $I_{12}$ is $C^1$ and in particular we deduce that
\begin{eqnarray}\label{F6}
\partial_f I_{12}(0,0)h(w)=-\alpha C_\alpha \fint_{\mathbb{T}}\frac{\hbox{Re}\Big((\overline\tau-\overline{w})\big(h(\tau)-h(w)\big)\Big)}{\vert \tau-w\vert^{2+\alpha}}d\tau.
\end{eqnarray}
Putting together \eqref{F001}, \eqref{F4} and \eqref{F6} we find  for $h\in X^\alpha$,
\begin{eqnarray}\label{F9}
\partial_f G_2(0,0)h(w) \nonumber &=&{C_\alpha}\hbox{Im}\bigg\{\overline{w}\,\mathop{{\fint}}_\mathbb{T}\frac{h'(\tau)d\tau}{\vert\tau-w\vert^\alpha}\bigg\}-\widehat \mu_\alpha \hbox{Im}\{h^\prime(w)\}\\
&-&\alpha C_\alpha\hbox{Im}\bigg\{\overline{w}\,\mathop{{\fint}}_\mathbb{T}\frac{\hbox{Re}\Big((\overline\tau-\overline{w})\big(h(\tau)-h(w)\big)\Big)}{\vert \tau-w\vert^{2+\alpha}}d\tau\bigg\}.
\end{eqnarray}
According to \eqref{F1}, \eqref{F_2} and \eqref{F9}  we get for any $\Omega\in \RR$ and $h\in X^\alpha$,
\begin{eqnarray*}
\partial_f G^\alpha (0,\Omega,0)h(w)&=&\partial_f G_1(0,\Omega,0)h(w)-\partial_f G_2(0,0)h(w)+\partial_f G_3(0,0)h(w)\\
&=&-{C_\alpha}\hbox{Im}\bigg\{\overline{w}\,\mathop{{\fint}}_\mathbb{T}\frac{h'(\tau)d\tau}{\vert\tau-w\vert^\alpha}\bigg\}+\widehat \mu_\alpha \hbox{Im}\{h^\prime(w)\}\\
\nonumber &+&\alpha C_\alpha\hbox{Im}\bigg\{\overline{w}\,\mathop{{\fint}}_\mathbb{T}\frac{\hbox{Re}\Big((\overline\tau-\overline{w})\big(h(\tau)-h(w)\big)\Big)}{\vert \tau-w\vert^{2+\alpha}}d\tau\bigg\}\\
&\equiv&\mathcal{L}_1h+\mathcal{L}_2 h+\mathcal{L}_3 h.
\end{eqnarray*}
\vspace{0,3cm}
Finally note  that the extension for $G$ is obtained by putting together  
 \eqref{split1}, \eqref{G_0X1},\eqref{G_2X1}, \eqref{I_2X}, \eqref{FF1},\eqref{F001} and \eqref{diff1}. Note that 
to obtain the point vortex pairs we write
 $$
 G^\alpha(0,\Omega,0)=\Big(\Omega d-\frac{\alpha C_\alpha}{2(2d)^{1+\alpha}}\Big) e_1
 $$
which implies that this is  a solution if and only if
 $$
 \Omega=\Omega_{sing}^\alpha\equiv \frac{\widehat{C}_\alpha}{(2d)^{2+\alpha}}\cdot
 $$
This means   that two  point vortices $\pi \delta_0$ and $\pi \delta_{2d}$  rotate uniformly about their center $(d,0)$ with the angular velocity $\Omega_{sing}^\alpha.$
\\
It remains to compute explicitly $\partial_f G^\alpha(0,\Omega,0)$ and show that it is an 
isomorphism from $X^\alpha$ to $Y^\alpha$. To this aim we will start  
computing $\partial_f G_2(0,0)$ whose expression is given \mbox{in \eqref{F9}. }
It is easy to see that
\begin{eqnarray*}
\alpha C_\alpha\overline{w}\mathop{{\fint}}_\mathbb{T}\frac{\hbox{Re}\Big((\overline\tau-\overline{w})\big(h(\tau)-h(w)\big)\Big)}{\vert \tau-w\vert^{2+\alpha}}d\tau&=&\frac12\alpha C_\alpha\overline{w}\mathop{{\fint}}_\mathbb{T}\frac{({\tau}-{w})(h(\overline \tau)-h(\overline w))}{\vert\tau-w\vert^{2+\alpha}} d\tau\\&+&\frac12 \alpha C_\alpha\overline{w}\mathop{{\fint}}_\mathbb{T}\frac{(\overline{\tau}-\overline{w})(h(\tau)-h(w))}{\vert\tau-w\vert^{2+\alpha}} d\tau\\
&\equiv&I_4(h(w))+I_5(h(w)).
\end{eqnarray*}
According \cite[p.360]{H-H} these terms were computed and take the form
\begin{eqnarray}\label{i4}
\hbox{I}_4(h(w))&=& \frac{\alpha(1+\frac{\alpha}{2})}{2(2-\alpha)} \frac{C_\alpha\Gamma(1-\alpha)}{\Gamma^2(1-\alpha/2)}\sum_{n\geq 1}a_n\bigg(1-\frac{\big(2+\frac{\alpha}{2}\big)_n}{\big(2-\frac{\alpha}{2}\big)_n}\bigg)w^{n+1}
\end{eqnarray}
and
\begin{eqnarray}\label{i5}
\hbox{I}_5 (h(w))&=&- \frac{\alpha C_\alpha \Gamma(1-\alpha)}{4\Gamma^2(1-\alpha/2)}\sum_{n\geq 1}a_n\bigg(1-\frac{\big(\frac{\alpha}{2}\big)_n}{\big(-\frac{\alpha}{2}\big)_n}\bigg)\overline{w}^{n+1}.
\end{eqnarray}
It follows that
\begin{eqnarray}\label{i6}
\mathcal{L}_3 h(w)=\textnormal{Im}\{I_4(h(w))+I_5(h(w))\}&=& \frac{\alpha C_\alpha \Gamma(1-\alpha)}{4\Gamma^2(1-\alpha/2)}\sum_{n\geq 1}a_n\beta_n e_{n+1}.
\end{eqnarray}
with
\begin{eqnarray}\label{i7}
\nonumber\beta_n&=& \bigg(1-\frac{\big(\frac{\alpha}{2}\big)_n}{\big(-\frac{\alpha}{2}\big)_n}\bigg)+\frac{1+\frac\alpha2}{1-\frac\alpha2}\bigg(1-\frac{\big(2+\frac{\alpha}{2}\big)_n}{\big(2-\frac{\alpha}{2}\big)_n}\bigg)\\
\nonumber &=& \bigg(1-\frac{\big(\frac{\alpha}{2}\big)_n}{\big(-\frac{\alpha}{2}\big)_n}\bigg)+\frac{1+\frac\alpha2}{1-\frac\alpha2}-\frac{\big(1+\frac{\alpha}{2}\big)_{n+1}}{\big(1-\frac{\alpha}{2}\big)_{n+1}}\\
&=&\frac{2}{1-\frac\alpha2}+\frac{\big(1+\frac{\alpha}{2}\big)_{n-1}}{\big(1-\frac{\alpha}{2}\big)_{n-1}}-\frac{\big(1+\frac{\alpha}{2}\big)_{n+1}}{\big(1-\frac{\alpha}{2}\big)_{n+1}}\cdot
\end{eqnarray}
Regarding  the first  term $\mathcal{L}_1(h(w))$ it may be rewritten in the form
$$
\mathcal{L}_1(h(w))=\hbox{Im}\big\{\hbox{I}_3(h(w))\big\}
$$
with
\begin{eqnarray*}
\hbox{I}_3(h(w)) &\equiv&-{C_\alpha}\overline{w}\fint_\mathbb{T}\frac{h'(\tau)}{\vert w-\tau\vert^\alpha}d\tau\notag\\ &=&{C_\alpha}\sum_{n\geq 1}na_n\overline{w}\fint_\mathbb{T}\frac{\overline{\tau}^{n+1}}{\vert w-\tau\vert^\alpha}d\tau.\notag\\ \end{eqnarray*}
Once again we get in view of \cite[p.360]{H-H} that
\begin{equation}\label{i21}
\hbox{I}_3(h(w))=\frac{C_\alpha\Gamma(1-\alpha)}{\Gamma^2(1-\alpha/2)}\sum_{n\geq 1}n a_n\frac{\big(\frac{\alpha}{2}\big)_n}{\big(1-\frac{\alpha}{2}\big)_n}\overline{w}^{n+1}.
\end{equation}
Therefore
\begin{equation}\label{i2}
\mathcal{L}_1(h(w))=-\frac{C_\alpha\Gamma(1-\alpha)}{\Gamma^2(1-\alpha/2)}\sum_{n\geq 1}n a_n\frac{\big(\frac{\alpha}{2}\big)_n}{\big(1-\frac{\alpha}{2}\big)_n}e_{n+1}.
\end{equation}
For $\mathcal{L}_2$ we readily get by \eqref{C_XX},
\begin{eqnarray*}
\mathcal{L}_2h(w)&=&\widehat \mu_\alpha\sum_{n\geq1} n a_n e_{n+1}\\
&=&\frac{\alpha C_\alpha\Gamma(1-\alpha)}{(2-\alpha)\Gamma^2(1-\frac\alpha2)} \sum_{n\geq1} n a_n e_{n+1}.
 \end{eqnarray*}
Putting together the preceding identities yields to
 \begin{eqnarray}\label{D-exp}
\nonumber\partial_f G^\alpha(0,\Omega,0)h(w)&=&\mathcal{L}_1h(w)+\mathcal{L}_2h(w)+\mathcal{L}_3 h(w)\\
&=&\frac{\alpha C_\alpha\Gamma(1-\alpha)}{4\Gamma^2(1-\frac\alpha2)}\sum_{n\geq1} a_n \gamma_n e_{n+1},
 \end{eqnarray}
 with
 \begin{eqnarray*}
\gamma_n&=&\beta_n-\frac{4}{\alpha}\frac{\big(\frac{\alpha}{2}\big)_n}{\big(1-\frac{\alpha}{2}\big)_n} n+\frac{4}{2-\alpha} n\\
&=&\beta_n-2n\frac{\big(1+\frac{\alpha}{2}\big)_{n-1}}{\big(1-\frac{\alpha}{2}\big)_n} +\frac{2n}{1-\frac\alpha2}\\
&=&\frac{2(1+n)}{1-\frac\alpha2}+\frac{\big(1+\frac{\alpha}{2}\big)_{n-1}}{\big(1-\frac{\alpha}{2}\big)_{n-1}}-
\frac{\big(1+\frac{\alpha}{2}\big)_{n+1}}{\big(1-\frac{\alpha}{2}\big)_{n+1}}-\frac{2n}{n-\frac\alpha2}
\frac{\big(1+\frac{\alpha}{2}\big)_{n-1}}{\big(1-\frac{\alpha}{2}\big)_{n-1}}\\
&=& \frac{2(1+n)}{1-\frac\alpha2}-\frac{\big(1+\frac{\alpha}{2}\big)_{n}}{\big(1-\frac{\alpha}{2}\big)_{n}}-\frac{\big(1+\frac{\alpha}{2}\big)_{n+1}}{\big(1-\frac{\alpha}{2}\big)_{n+1}}\cdot
 \end{eqnarray*}
\\
Now we shall prove that $\partial_f G^\alpha(0,\Omega,0):X^\alpha\to \widehat{Y}^\alpha$ is an isomorphism.  The case $\alpha=0$ is elementary since
$$
\partial_f G^0(0,\Omega,0)h(w)=\frac12\hbox{Im}(h^\prime(w))
$$
and one can easily check that this operator is an isomorphism from $X^0$ to $\widehat{Y}^0.$ So it remains to check the case $\alpha\in (0,1)$. Thus to verify that $\partial_f G^\alpha(0,\Omega,0)$  is   one-to-one   it is enough to prove  the following:  There exist two constants $C_1>0$ and $C_2>0$ such that for any $n\geq1$
\begin{equation}\label{Ineq1}
C_1 n\le \gamma_n\le C_2 n.
\end{equation}
It easy to check that 
\begin{eqnarray*}
\frac{(1+\frac{\alpha}{2})_n}{(1-\frac{\alpha}{2})_n}<\frac{n+\frac{\alpha}{2}}{1-\frac{\alpha}{2}}\beta,
\end{eqnarray*}
where $\beta=\frac{1+\frac{\alpha}{2}}{2-\frac{\alpha}{2}}<1.$
Therefore we deduce by simple computations that for $\alpha\in[0,1]$ 
\begin{eqnarray*}
\gamma_n>\frac{2(1+n)}{1-\frac{\alpha}{2}}-\beta \frac{n+\frac{\alpha}{2}}{1-\frac{\alpha}{2}}-\beta \frac{n+1+\frac{\alpha}{2}}{1-\frac{\alpha}{2}}\ge C_1(\alpha)n.
\end{eqnarray*}
On the other hand, we readily get 
$$
\gamma_n\le C_2(\alpha)n,
$$
and hence  the proof of \eqref{Ineq1} is achieved.  It remains to prove that $\partial_f G^\alpha(0,\Omega,0)$ is onto. Let $g\in \widehat{Y}^\alpha$, we shall prove that the equation $\partial_f G^\alpha(0,\Omega,0)h=g$ admits a solution $h\in X^\alpha.$ The functions $g$ and $h$ take the form
$$
g(w)=\frac{\widehat C_\alpha\Gamma(1-\alpha)}{4\Gamma^2(1-\frac\alpha2)}\sum_{n\geq 1} b_n e_{n+1}(w)\quad\hbox{and}\quad h(w)=\sum_{n\geq 1} a_n \overline{w}^n.
$$
 Therefore using \eqref{D-exp} and \eqref{Ineq1} the equation $\partial_f G^\alpha(0,\Omega,0)h=g$ is equivalent to
 $$
 a_n=\frac{b_n}{\gamma_n}, \, n\geq1.
 $$
 The only point to check is  $h\in C^{2-\alpha}(\mathbb{T})$, that is
 $$
 w\in\mathbb{T}\mapsto \sum_{n\geq1}\frac{b_n}{\gamma_n}\overline{w}^n\in C^{2-\alpha}(\mathbb{T}).
 $$
 From \cite[p.358]{H-H}, there exists a constant $C>0$ such that
 $$
\frac{ \big(1+\frac{\alpha}{2}\big)_{n}}{\big(1-\frac{\alpha}{2}\big)_{n}}= C n^\alpha+O\big(\frac{1}{n^{1-\alpha}}\big)
 $$
which implies in turn that
 $$
 \gamma_n= n\Big(\frac{2}{1-\frac\alpha2} -2C \frac{1}{n^{1-\alpha}}+O\big(\frac{1}{n^{2-\alpha}}\big)\Big).
 $$
 It is not difficult to show that $h\in L^\infty(\mathbb{T})$ and thus it remains to check that $h^\prime\in C^{1-\alpha}(\mathbb{T}).$ Note that
 $$
 -wh^\prime(w)=\sum_{n\geq1}\frac{n b_n}{\gamma_n} \overline{w}^{n}
 $$
 and
 \begin{eqnarray*}
 \frac{n b_n}{\gamma_n}&=&\frac{b_n}{\frac{2}{1-\frac\alpha2} -2C \frac{1}{n^{1-\alpha}}+O\big(\frac{1}{n^{2-\alpha}}\big)}\\
 &=&\frac{(1-\frac{\alpha}{2})b_n}{2}+ \frac{C(1-\frac{\alpha}{2})}{n^{1-\alpha}\big(\frac{2}{1-\frac\alpha2} -2C \frac{1}{n^{1-\alpha}}\big)}b_n+O\big(\frac{1}{n^{2-\alpha}}\big) b_n\\
 &\equiv &\frac{(1-\frac{\alpha}{2})b_n}{2}+\alpha_n b_n+O\big(\frac{1}{n^{2-\alpha}}\big) b_n.
 \end{eqnarray*}
 Using the continuity of  Szeg\"{o} projector $\Pi$ in $C^{1-\alpha}(\mathbb{T})$ we obtain easily that
 $$
 \tilde{h}: w\mapsto \sum_{n\geq1}b_n \overline{w}^n\in C^{1-\alpha}(\mathbb{T}).
 $$
 Define the kernels
 $$
 K_1(w)=\sum_{n\geq1} \alpha_n\overline{w}^n \quad\hbox{and}\quad K_2(w)=\sum_{n\geq1}O\big(\frac{1}{n^{2-\alpha}}\big)\overline{w}^n.
 $$
 The remainder term of $-wh^\prime$ is given by
 $$
 K_1\star\tilde{h}+K_2\star\tilde{h}.
 $$
As the kernel $K_2\in L^{\infty}(\mathbb{T})\subset L^1(\mathbb{T})$ then  $K_2\star \tilde{h}\in C^{1-\alpha}(\mathbb{T}).$ In \cite[p.363-366]{H-H} we established that $K_1\in L^1(\mathbb{T})$ and therefore we obtain $K_1\star \tilde{h}\in C^{1-\alpha}(\mathbb{T})$ and this gives finally $h^\prime\in C^{1-\alpha}(\mathbb{T})$ which concludes the proof.
 
\end{proof}

\subsection{Relationship between the angular velocity and the boundary shape}\label{real-velo}
 As we have seen in Proposition \ref{prop1} the linear
 operator $\partial_fG^\alpha(0,\Omega,0)$ is an isomorphism from $X^\alpha$ to $\widehat Y^\alpha$ 
 and not to the space $Y^\alpha.$ 
However the functional $G^\alpha$ has its range in $Y^\alpha$ which contains strictly 
$\widehat Y^\alpha$. The strategy will be to choose carefully $\Omega$ 
in such way that the range of $G^\alpha$ is contained in $\widehat Y^\alpha.$ This condition is
strong 
enough to uniquely determine $\Omega$ and by this way $\Omega$ plays the role of a Lagrangian multiplier with respect to the range constraint. The main result reads as follows.
\begin{proposition}\label{propX1}
 Let $\alpha\in [0,1).$ There exists a $C^1$ function $\mathcal{R}^\alpha:(-\frac12,\frac12)\times B_1^\alpha\to \RR$  such that, for 
 \begin{eqnarray*}
  \Omega=\Omega^\alpha(\EE,f)
=\Omega_{sing}^\alpha+\mathcal{R}^\alpha(\EE,f),
\end{eqnarray*}
 the modified function $\widehat{G}^\alpha:(-\frac12,\frac12)\times B_1^\alpha\to \widehat{Y}^\alpha$ given  by
 $$
\widehat{G}^\alpha(\EE,f)=G^\alpha\big(\EE,\Omega^\alpha(\EE,f),f\big)
 $$
is well-defined and is of class $C^1$. Moreover,
$$
\forall f\in B_1^\alpha,\,\mathcal{R^\alpha}(0,f)=0\quad\hbox{and}\quad \widehat{G}^\alpha(0,0)=0.
$$
Notice that $\Omega_{sing}^\alpha$ was defined in \eqref{Osing}.
 
 \end{proposition}
\begin{proof}
The proof will be separated in different parts: the case $\alpha=0$ and the case 
$\alpha\in(0,1)$.
\\
{\bf{Part $\hbox{I}$: case $\alpha=0$.}}  A sufficient and necessary condition to guarantee  that  $G^\alpha$ admits a range contained in  the space $\widehat{Y}^\alpha$ is  that its  first Fourier 
coefficients  vanishes. 
 This condition amounts to
 $$
 \fint_{\mathbb{T}} G^\alpha(\EE,\Omega, f(w)) dw=0
 $$
 or equivalently
\begin{equation}\label{const5}
\fint_{\mathbb{T}}F^\alpha\big(\Omega,\EE,f(w)\big)\big(\overline{w}^2-1\big) dw=0.
\end{equation}
We recall that $F^\alpha$ was defined in \eqref{def1} and \eqref{model}. For $\alpha=0$ one may use 
the residue theorem to get
$$
\fint_{\mathbb{T}}F_1\big(\Omega,\EE,f(w)\big)\overline{w}^2 dw=2\Omega\bigg(-d+\EE^3\fint_{\mathbb{T}}f(\overline{w})\overline{w}
f^\prime(w)dw\bigg)
$$
and
$$
\fint_{\mathbb{T}}F_1\big(\Omega,\EE,f(w)\big) dw=2\Omega\bigg(-d\EE\fint_{\mathbb{T}}w f^\prime(w) dw+\EE^3\fint_{\mathbb{T}}
f(\overline{w}){w}f^\prime(w)dw\bigg).
$$
This last identity can be written in the form
$$
\fint_{\mathbb{T}}F_1\big(\Omega,\EE,f(w)\big) dw
=2\Omega\bigg(d\EE\fint_{\mathbb{T}}f(w) dw+\EE^3\fint_{\mathbb{T}}f(\overline{w}){w}f^\prime(w)dw\bigg).
$$
Consequently
\begin{eqnarray*}
\fint_{\mathbb{T}}F_1\big(\Omega,\EE,f(w)\big)\big(\overline{w}^2-1\big) dw=2\Omega \bigg(-d\Big(1+\EE\fint_{\mathbb{T}}f(w) dw\Big)+\EE^3\fint_{\mathbb{T}}f(\overline{w})f^\prime(w)\big(\overline{w}-w\big)dw\bigg).
\end{eqnarray*}
Now we shall evaluate the contribution of $F_3.$ First, observe that
$$
F_3(\EE,f(w))=-\widehat{F_3}(w) w\big(1+\EE f^\prime(w)\big),$$
with the notation
$$ \widehat{F_3}(\EE,f(w))\equiv\fint_{\mathbb{T}}\frac{\overline{\tau}+\EE{f(\overline\tau)}}{\EE(\tau+w)+\EE^2 \big(f(\tau)+f(w)\big)-2d}\big(1+\EE f^\prime(\tau)\big)d\tau.
$$
We write
\begin{eqnarray*}
\frac{\overline{\tau}+\EE{f(\overline\tau)}}{\EE(\tau+w)+\EE^2 \big(f(\tau)+f(w)\big)-2d}&=&-\frac{\overline{\tau}}{2d}+\EE\frac{f(\overline{\tau})}{\EE(\tau+w)+\EE^2 \big(f(\tau)+f(w)\big)-2d}\\
&+&\frac{\EE}{2d}\frac{\tau+w+\EE(f(\tau)+f(w))}{\EE(\tau+w)+\EE^2 \big(f(\tau)+f(w)\big)-2d}\overline{\tau}\\
&\equiv&-\frac{\overline{\tau}}{2d}+\EE g_3(\EE,\tau,w).
 \end{eqnarray*}
 Thus
 $$
  \widehat{F_3}(\EE,f(w))=-\frac{1}{2d}+\EE\fint_{\mathbb{T}}g_3(\EE,\tau,w)\big(1+\EE f^\prime(\tau)\big) d\tau.
 $$
 Hence
\begin{eqnarray*}
 \fint_{\mathbb{T}}F_3\big(\Omega,\EE,f(w)\big)\overline{w}^2 dw&=&\frac{1}{2d}-\EE\fint_{\mathbb{T}}\fint_{\mathbb{T}}g_3(\EE,\tau,w)\big(1+\EE f^\prime(\tau)\big)\overline{w}\big(1+\EE f^\prime(w)\big) d\tau dw\\
 \fint_{\mathbb{T}}F_3\big(\Omega,\EE,f(w)\big) dw&=&-\frac{\EE}{2d}\fint_{\mathbb{T}} f(\tau) d\tau-\EE\fint_{\mathbb{T}}\fint_{\mathbb{T}}g_3(\EE,\tau,w)\big(1+\EE f^\prime(\tau)\big){w}\big(1+\EE f^\prime(w)\big) d\tau dw.
\end{eqnarray*}
Consequently
\begin{eqnarray*}
 \fint_{\mathbb{T}}F_3\big(\Omega,\EE,f(w)\big)\big(\overline{w}^2-1\big) dw&=&\frac{1}{2d}+\frac{\EE}{2d}\fint_{\mathbb{T}} f(\tau) d\tau\\
 &-&\EE\fint_{\mathbb{T}}\fint_{\mathbb{T}}g_3(\EE,\tau,w)\big(1+\EE f^\prime(\tau)\big)(\overline{w}-w)\big(1+\EE f^\prime(w)\big) d\tau dw.
 \end{eqnarray*}
 On the other hand using residue  theorem we get
 \begin{eqnarray*}
 F_2(\EE,f(w))&=&\EE \fint_{\mathbb{T}}\frac{A\overline{B}-\overline{A} B}{A(A+\EE B)} f^\prime(\tau)d\tau\, w\big(1+\EE f^\prime(w)\big)\\
 &+&\EE\fint_{\mathbb{T}}\frac{\big(\overline{A}B-A\overline{B}\big) B }{A^2(A+\EE B)}d\tau\, w\big(1+\EE f^\prime(w)\big)\\
 &\equiv&\EE g_2(\EE,w)w\big(1+\EE f^\prime(w)\big).
  \end{eqnarray*}
 Therefore
 \begin{equation*}
 \fint_{\mathbb{T}}F_2\big(\Omega,\EE,f(w)\big)\big(\overline{w}^2-1\big) dw=\EE\fint_{\mathbb{T}}g_2(\EE,w)(\overline{w}-w)\big(1+\EE f^\prime(w)\big)  dw.
 \end{equation*}
The equation \eqref{const5} becomes
\begin{eqnarray*}
 2\Omega \bigg(d\Big[1+\EE\fint_{\mathbb{T}}f(w) dw\Big]&-&\EE^3\fint_{\mathbb{T}}f(\overline{w})f^\prime(w)\big(\overline{w}-w\big)dw\bigg)=\frac{1}{2d}+\frac{\EE}{2d}\fint_{\mathbb{T}} f(\tau) d\tau\\&+&\EE\fint_{\mathbb{T}}g_2(\EE,w)(\overline{w}-w)\big(1+\EE f^\prime(w)\big)  dw\\
 &+&\EE\fint_{\mathbb{T}}\fint_{\mathbb{T}}g_3(\EE,\tau,w)\big(1+\EE f^\prime(\tau)\big)(w-\overline{w})\big(1+\EE f^\prime(w)\big) d\tau dw\\
 &\equiv&\frac{1}{2d}+\frac{\EE}{2d} T_1(\EE,f)
  \end{eqnarray*}
  which can be written in the form
 \begin{eqnarray}\label{const4}
 \nonumber \Omega&=&\Omega^0(\EE,f)\\
 \nonumber&=&\frac{1}{4d^2}\frac{1+\EE T_1(\EE,f)}{1-\EE T_2(\EE,f)}\\
  \nonumber&=&\Omega_{sing}^0+\frac{\EE}{4d^2}\frac{T_1(\EE,f)+T_2(\EE,f)}{1-\EE T_2(\EE,f)}\\
  &\equiv&\Omega_{sing}^0+\mathcal{R}^0(\EE,f),
  \end{eqnarray}
  where 
  $$
  T_2(\EE,f)=-\fint_{\mathbb{T}}f(w) dw+\frac{\EE^2}{d}\fint_{\mathbb{T}}f(\overline{w})f^\prime(w)\big(\overline{w}-w\big)dw.
  $$
  Now we intend to prove that $(\EE,f)\mapsto \Omega^0(\EE,f)$ is $C^1$. For this aim 
it is enough to check that the functions $(\EE,f)\mapsto T_1(\EE,f)$ and $(\EE,f)\mapsto T_2(\EE,f)$ are $C^1$ functions and that $|T_2(\EE,f)|<2.$ 
Since $f$ has real coefficients it is clear that  $T_2(\EE,f)\in \mathbb R$ and 
$$
|T_2(\EE,f)|\le \| f \|_{C^{1+\beta}(\mathbb T)}+\frac{\EE^2}{d}2\| f \|^2_{C^{1+\beta}(\mathbb T)}<2.
$$
On the other hand, $T_2$ is polynomial in the variables $\EE, \, f$ and $f^\prime$ and so it should be 
 a $C^1$ function from $(\frac12,\frac12)\times B_1^0$ to $\mathbb R.$
Let's take now the functional 
\begin{eqnarray*}
T_1(\EE,f)&=&\fint_{\mathbb T}f(\tau)d\tau+2d\fint_{\mathbb T}g_2(\EE,w)(\overline w-w)(1+\EE f^\prime(w))dw\\
&+&2d\fint_{\mathbb T}\fint_{\mathbb T}g_3(\EE,\tau,w)(1+\EE f^\prime(\tau))(w-\overline w)(1+\EE f^\prime(\tau))d\tau dw, 
\end{eqnarray*}
where 
$$
g_2(\EE,f)=\fint_{\mathbb T}\frac{A\overline B-\overline A B}{A(A+\EE B)}f^\prime(\tau)d\tau+\fint_{\mathbb T}
\frac{(\overline AB-A \overline B)B}{A^2(A+\EE B)}d\tau,
$$
with  $A=\tau-w, B=f(\tau)-f(w)$ and
$$
g_3(\EE,f)=\frac{f(\overline \tau)}{\EE(\tau+w)+\EE^2(f(\tau)+f(w))-2d}+2d \frac{\tau+w+\EE (f(\tau)+f(w))}
{\EE(\tau+w)+\EE^2(f(\tau)+f(w))-2d}\overline \tau.
$$
Since $|\EE |<\frac12$ and $\|f\|_{C^{1+\beta}}<1$ we get that $g_3$ is a bounded function. Moreover
\begin{eqnarray*}
|g_2(\EE,f)(w)|&\le& 2 \int_{\mathbb T}\Big|\frac{\textnormal{Im}\{(\tau-w)(f(\overline \tau)-
f(\overline w))\}}{(\tau -w)(\tau -w+\EE(f(\tau)-f(w)))} f^\prime(\tau)\Big||d\tau|\\
&+&2\int_{\mathbb T}\Big|\frac{\textnormal{Im}\{(\overline \tau -\overline w)(f(\tau)-f(w))\}(f(\tau)-f(w))}{(\tau -w)^2(\tau -w+
\EE(f(\tau)-f(w)))}\Big||d\tau|\le C,
\end{eqnarray*}
where in the last inequality we use again that $|\EE |<\frac12$ and $\|f\|_{C^{1+\beta}}<1.$
To prove that $T_1$ is a $C^1$ function it suffices to check that the partial derivatives
 of $g_2(\EE, f)$ and $g_3(\EE, f)$ are continuous functions on $(-\frac 12, \frac 12)\times B_1^0.$ 
 Observe that,

\begin{eqnarray*} 
\partial_{\EE}g_2(\EE, f)=&-&2i\fint_{\mathbb T}\frac{\textnormal{Im}\{( \tau- w)(f(\overline \tau)-
f(\overline w))\}}{(\tau-w)(\tau-w+\EE(f(\tau)-f(w)))^2}(f(\tau)-f(w))f^\prime(\tau)d\tau\\
&-&2i\fint_{\mathbb T}\frac{\textnormal{Im}\{(\overline \tau-\overline w)(f(\tau)-f(w))\}}
{(\tau-w)^2(\tau-w+\EE(f(\tau)-f(w)))^2}(f(\tau)-f(w))^2d\tau.
\end{eqnarray*}
It is easy to verify that the kernels involved in the above integral operators satisfy the conditions
 of Lemma \ref{noyau} and so we can conclude that $\partial_{\EE}g_2(\EE, f)$ is a continuous function 
from $(-\frac 12, \frac12)\times B_1^0$ to $\mathbb R.$
For any direction  $h \in X^0$ straightforward computations yield 
\begin{eqnarray*}
\partial_{f}g_2(\EE, f)(h)&=&2i\fint_{\mathbb T}\frac{\textnormal{Im}\{( \tau- w)(h(\overline \tau)
-h(\overline w))\}}{(\tau-w)(\tau-w+\EE(f(\tau)-f(w)))} f^\prime(\tau)d\tau\\
&+&2i\fint_{\mathbb T}\frac{\textnormal{Im}\{( \tau- w)(f(\overline \tau)-f(\overline w))\}}{(\tau-w)(\tau-w+\EE(f(\tau)-f(w)))}
h^\prime(\tau)d\tau\\
&-&2i\EE \fint_{\mathbb T}\frac{\textnormal{Im}\{( \tau- w)(f(\overline \tau)-f(\overline w))\}}{(\tau-w)(\tau-w+\EE(f(\tau)-f(w)))^2}
(h(\tau)-h( w))f^\prime(\tau)d\tau\\
&+&2i\fint_{\mathbb T}\frac{\textnormal{Im}\{(\overline \tau-\overline w)(h(\tau)-h(w))\}}
{(\tau-w)^2(\tau-w+\EE(f(\tau)-f(w)))}(f(\tau)-f(w))d\tau\\
&+&2i\fint_{\mathbb T}\frac{\textnormal{Im}\{(\overline \tau-\overline w)(f(\tau)-f(w))\}}{(\tau-w)^2(\tau-w+\EE(f(\tau)-f(w)))}
(h(\tau)-h(w))d\tau.\\
&-&2i\EE\fint_{\mathbb T}\frac{\textnormal{Im}\{(\overline \tau-\overline w)(f(\tau)-f(w))\}}{(\tau-w)^2(\tau-w+\EE(f(\tau)-f(w)))^2}
(f(\tau)-f(w))(h(\tau)-h(w))d\tau.
\end{eqnarray*}
Again the kernels involved in the integral operators satisfy the conditions in Lemma \ref{noyau} 
and so $\partial_{f}g_{2}(\EE, f)(h)$ defines a continuous  function from 
$(-\frac 12, \frac 12)\times B_1^0$ to $\mathbb R.$
Reproducing similar  computations one can prove that $g_3(\EE, f)$ is a $C^1$ function from 
$(-\frac 12, \frac 12)\times B_1^0$ to $\mathbb R,$ and so we get that the function 
$\Omega^0(\EE,f)$ is $C^1.$
\vspace{0,3cm}
\\
{\bf{Part $\hbox{II}$: case $\alpha\in(0,1)$.}}
 As in the first part of the proof, the condition imposed to $\Omega$ 
to guarantee that    the component of $e_1$ in the Fourier expansion  $G^\alpha$ vanishes is  
\begin{equation*}
\fint_{\mathbb{T}}G^\alpha\big(\EE,\Omega,f(w)\big)dw=0.
\end{equation*}
According to the  decomposition \eqref{split1} this assumption is equivalent to

\begin{equation}\label{restr1}
A_1=A_2-A_3,
\end{equation}
with 
$$
A_j=-2i\fint_{\mathbb{T}}G_j\big(\EE,\Omega,f(w)\big) dw.
$$
Note that $A_j$ is the Fourier coefficient of $e_1=\hbox{Im}(w)$ in $G_j$ and when $G_j=\hbox{Im}(F_j)$ then
$$
A_j=\fint_{\mathbb{T}}F_j\big(\EE,\Omega,f(w)\big)\big(\overline{w}^2-1\big) dw.
$$
Looking for  the Fourier expansion of $\hbox{Im }\Big\{\Omega\Big(\EE w-d\Big)\, \overline{w}\,\Big(1+\EE|\EE|^\alpha\overline{f^\prime(w)}\Big)\Big\}$ we deduce that the coefficient of $e_1$ is 
\begin{eqnarray*}
\Omega d\Big(1+\EE|\EE|^\alpha\fint_{\mathbb{T}}f(\tau) d\tau\Big).
\end{eqnarray*}
In a similar way the Fourier  coefficient of $e_1$ in $\Omega\EE^2 |\EE|^{\alpha}\hbox{Im}\Big(f(w)\overline{w}\big(1+\EE|\EE|^\alpha\overline{f^\prime(w)}\big)\Big)$ is 
$$
\Omega\EE ^3 |\EE|^{2\alpha}\fint_{\mathbb{T}}f(w)\overline{f^\prime(w)} \big(\overline{w}^3-\overline{w}\big)dw.
$$
Consequently,
\begin{eqnarray}\label{A_0}
\nonumber A_1&=&
\Omega d\Big(1+\EE|\EE|^\alpha\fint_{\mathbb{T}}f(\tau) d\tau\Big)\\
\nonumber&+&\Omega\EE ^3 |\EE|^{2\alpha}\fint_{\mathbb{T}}f(w)\overline{f^\prime(w)} \big(\overline{w}^3-\overline{w}\big)dw\\
&\equiv&\Omega d+\Omega \EE|\EE|^\alpha\mathcal{R}_0(\EE,f)
\end{eqnarray}
with $\mathcal{R}_0:(-\frac12,\frac12)\times B_1^0\to \RR.$  We point out that for any $(\EE,f)\in (-\frac12,\frac12)\times B_1^0$
\begin{eqnarray}\label{Est-R}
\nonumber |\mathcal{R}_0(\EE,f)|&\le& d\|f\|_{L^\infty}+2|\EE|^{2+\alpha}\|f\|_{L^\infty}\|f^\prime\|_{L^\infty}\\
&\le&d+2|\EE|^{2+\alpha}
\end{eqnarray}
which  means that the function $\mathcal{R}_0$ is well-defined. It is clear that it is differentiable with continuity in the $\EE$-variable and moreover the function is polynomial in $f$ and $f^\prime$ and so its derivative satisfies the required assumption. Therefore we conclude that $\mathcal{R}_0$ is a $C^1$ function from $(-\frac12,\frac12)\times B_1^\alpha$ to $\RR$.
Now we shall compute the quantity $A_1$ associated to $G_1$ which is described by  \eqref{F001} and \eqref{FF1}. Thus
 \begin{eqnarray*}
\nonumber -A_2&=& \fint_{\mathbb{T}}I_{11}(\EE,f(w))\big(\overline{w}-\overline{w}^3\big)dw+\fint_{\mathbb{T}}I_{12}(\EE,f(w))\big(\overline{w}-\overline{w}^3\big)dw\\
\nonumber &+&\EE|\EE|^\alpha \fint_{\mathbb{T}}\Big[I_{11}(\EE,f(w))+I_{12}(\EE,f(w))\Big]\overline{f^\prime(w)}\big(\overline{w}-\overline{w}^3\big)dw\\
&\equiv&A_{11}+A_{12}+\EE|\EE|^\alpha A_{13}.
\end{eqnarray*}
To calculate $A_{11}$ we use  \eqref{FF1} which yields
\begin{eqnarray}\label{F0001}
\nonumber A_{11}&=& C_\alpha \fint_{\mathbb{T}}\fint_{\mathbb{T}}\frac{f^\prime(\tau) \big(\overline{w}-\overline{w}^3\big)}{|\tau-w|^\alpha}dw d\tau+C_\alpha \fint_{\mathbb{T}}\fint_{\mathbb{T}}{f^\prime(\tau) \big(\overline{w}-\overline{w}^3\big)}K(\tau,w)dw d\tau
\end{eqnarray}
with
$$
K(\tau,w)=\frac{1}{|\phi(\tau)-\phi(w)|^\alpha}-\frac{1}{|\tau-w|^\alpha},\quad \phi(w)
=w+\EE|\EE|^\alpha f(w).
$$
From  the Fourier expansion \eqref{i21} we conclude that
$$
\fint_{\mathbb{T}}\fint_{\mathbb{T}}\frac{f^\prime(\tau) \big(\overline{w}-\overline{w}^3\big)}{|\tau-w|^\alpha}dw d\tau=0
$$
and therefore
\begin{eqnarray*}
\nonumber A_{11}&=&C_\alpha \fint_{\mathbb{T}}\fint_{\mathbb{T}}{f^\prime(\tau) \big(\overline{w}-\overline{w}^3\big)}K(\tau,w)dw d\tau.
\end{eqnarray*}
According to \eqref{diff1} we get
\begin{eqnarray}\label{FX1}
 A_{11}&=&C_\alpha \EE|\EE|^\alpha\fint_{\mathbb{T}}\fint_{\mathbb{T}}{f^\prime(\tau) \big(\overline{w}-\overline{w}^3\big)}K(\EE,\tau,w)dw d\tau.
\end{eqnarray}
For the term $A_{12}$ recall  from \eqref{FF1} that
$$
I_{12}(\EE,f(w))=C_\alpha\fint_{\mathbb{T}} K(\EE,\tau,w)d\tau.
$$
Combining \eqref{diff1} with \eqref{Ta1} we get
\begin{eqnarray*}
\nonumber K(\EE,\tau,w)&=&-\alpha\frac{\hbox{Re}\Big((\overline\tau-\overline{w})\big(f(\tau)-f(w)\big)\Big)}{\vert \tau-w\vert^{2+\alpha}}\\
&+&\alpha(2+\alpha)\EE|\EE|^\alpha\int_0^1\int_{0}^1\frac{\widehat{K}(t,\EE,\tau,w)}{\vert \tau-w+st\EE|\EE|^\alpha(f(\tau)-f(w))\vert^{4+\alpha}}dtds\\
&-&\alpha\EE|\EE|^\alpha\int_{0}^1\frac{t\vert f(\tau)-f(w)\vert^{2}}{\vert \tau-w+t\EE|\EE|^\alpha(f(\tau)-f(w))\vert^{2+\alpha}}dt,
\end{eqnarray*}
with 
$$
\widehat{K}(t,\EE,\tau,w)\equiv\hbox{Re}\Big((\overline\tau-\overline{w})
\big(f(\tau)-f(w)\big)\Big)\Big[t\hbox{Re}\Big((\overline\tau-\overline{w})\big(f(\tau)-f(w)\big)\Big)+
s{t^2}\EE|\EE|^\alpha|f(\tau)-f(w)|^2\Big].
$$
Thus
$$
I_{12}(\EE,f(w))=-\alpha C_\alpha\fint_{\mathbb{T}}\frac{\hbox{Re}\Big((\overline\tau-\overline{w})\big(f(\tau)-f(w)\big)\Big)}{\vert \tau-w\vert^{2+\alpha}}d\tau+\EE|\EE|^\alpha\mathcal{I}_{12}(\EE,f(w)).
$$
Using  that the function $\mathcal{I}_{12}$ is a sum of terms defined by integral operators and those operators have kernels satisfying the conditions of Lemma \ref{noyau}, we can conclude that   ${I}_{12}:(-\frac12,\frac12)\times B_1^\alpha\to Y^\alpha$ is a $C^1$ function.  Now using \eqref{i6} we deduce that 
$$
\fint_{\mathbb{T}}\fint_{\mathbb{T}}\frac{\hbox{Re}\Big((\overline\tau-\overline{w})\big(f(\tau)-f(w)\big)\Big)}{\vert \tau-w\vert^{2+\alpha}}(\overline{w}-\overline{w}^3)d\tau dw=0
$$
which implies that
 \begin{eqnarray*}
\nonumber A_{12}&=& \fint_{\mathbb{T}}I_{12}(\EE,f(w))\big(\overline{w}-\overline{w}^3\big)dw\\
&=&\EE|\EE|^\alpha\fint_{\mathbb{T}}\mathcal{I}_{12}(\EE,f(w))\big(\overline{w}-\overline{w}^3\big)dw.
\end{eqnarray*}
Finally we get
\begin{eqnarray}\label{A_1}
\nonumber  -A_2&=&A_{11}+A_{12}+\EE|\EE|^\alpha A_{13}\\
\nonumber &=& C_\alpha \EE|\EE|^\alpha\fint_{\mathbb{T}}\fint_{\mathbb{T}}{f^\prime(\tau) \big(\overline{w}-\overline{w}^3\big)}K(\EE, \tau,w)dw d\tau\\
\nonumber &+&\EE|\EE|^\alpha\fint_{\mathbb{T}}\mathcal{I}_{12}(\EE,f(w))\big(\overline{w}-\overline{w}^3\big)dw+\EE|\EE|^\alpha A_{13}\\
&\equiv&-\EE|\EE|^\alpha\mathcal{R}_1(\EE,f).
\end{eqnarray}
Analyzing carefully all the terms  in $A_2$, as in the foregoing cases,  one may conclude that
$\mathcal{R}_1:(-\frac12,\frac12)\times B_1^\alpha\to \RR$ is $C^1$. So, it remains to compute
{$A_3$} which is described  by \eqref{restr1} \mbox{and \eqref{G_2}},
 \begin{eqnarray}\label{A_2}
\nonumber {A_3}&=&\fint_{\mathbb{T}}{F_3}(\EE,f(w))(\overline w^2-1)dw\\
\nonumber &=&\fint_{\mathbb T}\frac{\alpha C_\alpha}{2(2d)^{1+\alpha}}(\overline w^3-\overline w)dw-\EE\mathcal{R}_2(\EE,f)\\
&=&-\frac{\alpha C_\alpha}{2(2d)^{1+\alpha}}-\EE\mathcal{R}_2(\EE,f)
\end{eqnarray}
and we can check that $\mathcal{R}_2:(-\frac12,\frac12)\times B_1^\alpha\to \RR$ is well-defined and  $C^1$. Combining \eqref{restr1} with  \eqref{A_0}, \eqref{A_1} and \eqref{A_2} we deduce that,
$$
\Omega\Big(d+\EE|\EE|^\alpha\mathcal{R}_0(\EE,f)\Big)=\frac{\alpha C_\alpha}{2(2d)^{1+\alpha}}+\EE|\EE|^\alpha\mathcal{R}_1(\EE,f)+\EE \mathcal{R}_2(\EE,f).
$$
According to \eqref{Est-R}, since $d>2$ then for any $(\EE,f)\in (-\frac12,\frac12)\times B_1^\alpha$ we obtain
\begin{eqnarray*}
d+\EE|\EE|^\alpha\mathcal{R}_0(\EE,f)&\geq& d-d|\EE|^{1+\alpha}-2|\EE|^{3+2\alpha}\\
&\geq&\frac{d}{4}\cdot
\end{eqnarray*}
Therefore
\begin{eqnarray*}
\Omega&=&\Omega(\EE,f)\\
&=&\frac{\frac{\alpha C_\alpha}{2(2d)^{1+\alpha}}+\EE|\EE|^\alpha\mathcal{R}_1(\EE,f)+\EE \mathcal{R}_2(\EE,f)}{d+\EE|\EE|^\alpha\mathcal{R}_0(\EE,f)}\\
&\equiv&\Omega_{sing} +\mathcal{R}^\alpha(\EE,f),\qquad\Omega_{sing}= \frac{\widehat C_\alpha}{(2d)^{2+\alpha}},
\end{eqnarray*}
where  $ \mathcal{R}^\alpha:(-\frac12,\frac12)\times B_1^\alpha\to \RR$ is $C^1$ because it is obtained as an algebraic combination of $C^1$ functions without zeros in the denominator.  Obviously $\mathcal{R}^\alpha(0,f)=0 $ for any $f\in B_1^\alpha$ and this achieves the proof of the proposition.

\end{proof}

\subsection{Proof of the Main Theorem-$\hbox{(i)}$}\label{Secw1}
In this section we will give a precise statement of the first part of the main 
theorem which describes the structure of the solution in a neighborhood of the point vortex pairs.
Recall from Proposition \ref{propX1} that the existence of solutions to the V-states equation can be 
transformed into the resolution of
$$
\widehat{G}^\alpha(\EE,f)=0, \quad (\EE,f)\in (-\frac12,\frac12)\times B_1^\alpha
$$
with $\widehat{G}^\alpha$ being the functional defined by
$$
\widehat{G}^\alpha\big(\EE,f(w)\big)=G^\alpha\big(\EE,\Omega^\alpha(\EE,f),f\big)
$$
and $\Omega(\EE,f)$ has been introduced in Proposition \ref{propX1}.
The main result is the following.
\begin{proposition}\label{propY1}
Let $\alpha\in [0,1)$, then the following holds true.
\begin{enumerate}
\item The linear operator $\partial_f \widehat G^\alpha(0,0):X^\alpha\to \widehat{Y}^\alpha$ is an isomorphism and
$$
 \partial_f\widehat{G}^\alpha(0,0)h(w)=\sum_{n\geq1} a_n \widehat\gamma_n e_{n+1} 
 $$
  with
$$
\widehat \gamma_n=\frac{\widehat{C}_\alpha\Gamma(1-\alpha)}{4\Gamma^2(1-\frac\alpha2)}
\bigg( \frac{2(1+n)}{1-\frac\alpha2}-\frac{\big(1+\frac{\alpha}{2}\big)_{n}}{\big(1-\frac{\alpha}{2}\big)_{n}}
-\frac{\big(1+\frac{\alpha}{2}\big)_{n+1}}{\big(1-\frac{\alpha}{2}\big)_{n+1}}\bigg).
$$
\item There exists $\EE_0>0$ such that the set 
$$
\Big\{(\EE,f)\in [-\EE_0,\EE_0]\times B_1^\alpha,\, s.t.\quad \widehat{G}^\alpha(\EE,f)=0\Big\}
$$
is parametrized by one-dimensional curve $\EE \in[-\EE_0,\EE_0]\mapsto (\EE, f_\EE)$ and
$$
\forall \EE\in [-\EE_0,\EE_0]\backslash \{0\},\, f_\EE\neq0.
$$

\item If $(\EE,f)$ is a solution then $(-\EE,\tilde{f})$ is also a solution,  where
$$
 \tilde{f}(w)=f(-w), \quad \forall w\in \mathbb{T}.
$$
\item For all $\EE \in[-\EE_0,\EE_0]\backslash \{0\},$ the domain $D_1^\EE$ is strictly convex.
\end{enumerate}

\end{proposition}

\begin{proof}
{\bf{(i)}} From the composition rule
$$
\partial_f\widehat{G}^\alpha(0,0)h(w)=\partial_\Omega^\alpha G^\alpha(0,\Omega_{sing}^\alpha,0)\,\partial_f \Omega^\alpha(0,0) h(w)+\partial_f G^\alpha(0,\Omega_{sing}^\alpha,0)h(w).
$$
By virtue of  the expansion in $\EE$ of $\Omega^\alpha(\EE,f)$ given in Proposition \ref{propX1} we deduce that
\begin{eqnarray*}
\partial_f\Omega^\alpha(0,0)&=&\frac{d}{dt}\Omega^\alpha(0, th(w))_{|t=0}\\
&=&0
\end{eqnarray*}
and therefore
$$
\partial_f\widehat{G}^\alpha(0,0)h(w)=\partial_f G^\alpha(0,\Omega_{sing}^\alpha,0)h(w).
$$
Combining this identity once again with Proposition \ref{prop1} we deduce the desired result.
\\
{\bf{(ii)}} As we have seen before in Proposition \ref{propX1}, $\widehat{G}^\alpha:(-\frac12,\frac12)\times B_1^\alpha\to \widehat Y^\alpha$ is $C^1$ and  the linear operator  $\partial_f\widehat{G}^\alpha(0,0): X^\alpha\to \widehat{Y}^\alpha$ is invertible. Therefore we can conclude using the implicit function theorem. It remains to check that $f_\EE\neq0$ for $\EE\neq0.$ For this purpose,  we will prove that for any $\varepsilon$ small enough  and for  any $\Omega$ we can not get a solution with $f=0.$
So, it means that for $\EE\neq0$ we should get
$$
G^\alpha(\varepsilon, \Omega, 0)\neq 0.
$$
We shall start with the case $\alpha=0$ which is much more simpler. It is easy to check from \eqref{def1}  that 
$$F_1(\varepsilon,\Omega,0)=2\Omega\big(\EE-d \,w\big) \quad \hbox{ and}\quad  F_2(\varepsilon,0)=0.
$$ 
However  to compute  $F_3$ we proceed by Taylor expansion as follows,
\begin{eqnarray*}
F_3(\varepsilon,0)&=&-w\fint_{\mathbb T}\frac{\overline\tau}{\varepsilon(\tau+w)-2d}d\tau\\
&=&w\sum_{n\in\NN}\frac{\EE^n}{(2d)^{n+1}}\fint_{\mathbb{T}}\overline{\tau}(\tau+w)^n \, d\tau\\
&=&\sum_{n\in\NN}\frac{\EE^n}{(2d)^{n+1}} w^{n+1},
\end{eqnarray*}
which gives in turn
\begin{equation}\label{F3}
F_3(\EE,0)=\frac{w}{2d-\EE w}\cdot
\end{equation}
Consequently 
$$
G^0(\varepsilon, \Omega, 0)= \hbox{Im}\Big\{-2d\Omega w+\frac{w}{2d-\EE w}\Big\}. 
$$
and this quantity is not zero if $\varepsilon \neq 0$ is small enough. 
\\
Let us now move  to the case $\alpha\in (0,1)$. One finds using   \eqref{split1},  \eqref{G_0X1},  \eqref{F001} and \eqref{I12}, that
$$
G_1(\EE,\Omega,0)=-\Omega d\, \textnormal{Im}(\overline w)\quad\hbox{and} \quad G_2(\EE,\Omega,0)=0.
$$
To compute $G_3(\EE,0)$ it is enough  to calculate $I_2(\EE,0)$ because  $I_1(\varepsilon, 0)=0.$
The exact  computations turn out to be much more involved.  Thus we shall give the expansion of $I_2(\EE,0)$ at the order one in $\EE.$  Applying  \eqref{I_2XX} one gets
\begin{eqnarray*}
I_2(\EE,0)&=&\frac{\alpha C_\alpha}{2(2d)^{1+\alpha}}-\frac{\alpha C_\alpha\EE}{2(2d)^{2+\alpha}}\fint_{\mathbb{T}}|\tau+w|^2d\tau\\
&+&\frac{\alpha C_\alpha(2+\alpha)\EE}{2(2d)^{2+\alpha}}\fint_\mathbb{T}\Big(\hbox{Re}\big[\tau+w\big]\Big)^2d\tau+\EE^2 O(1)\\
&=& \Omega^\alpha_{sing} d-\frac{\widehat{C}_\alpha\EE}{2(2d)^{2+\alpha}} w+ \frac{\widehat{C}_\alpha(2+\alpha)\EE}{4(2d)^{2+\alpha}}\big(w+\overline{w})+\EE^2 O(1),
\end{eqnarray*}
and so 
$$
G_3(\varepsilon, 0)=\textnormal{Im}\{I_2(\EE,0)\overline w\}.
$$
Therefore the V-states equations becomes
\begin{eqnarray*}
\hbox{Im}\Big\{\big(I_2(\EE,0)-\Omega  d\big)\overline{w}\Big\}&=&d(\Omega_{sing}^\alpha-\Omega)\textnormal{Im}(\overline w)
+\EE \frac{\widehat{C}_\alpha(2+\alpha)}{4(2d)^{2+\alpha}}\hbox{Im}(\overline{w}^2)+\EE^2 O(1)\\
&=&0
\end{eqnarray*}
and this equation is impossible for  $0<\EE\le \EE_0$ with $\EE_0$ small enough depending on $d$ \mbox{and $\alpha$. }
\\
\\
{\bf{(iii)}} We shall only present the proof for the case $\alpha=0$ and for the same proof works as well for  $\alpha\in (0,1).$
Using the definition of $\tilde f$ one can check that $T_i(-\EE,\tilde f)=-T_i(\EE,f), \, \text{ for } i=1,2$ and so by (\ref{const4}) we obtain that 
$$
\Omega(\EE, f)=\Omega(-\EE, \tilde f).
$$
Taking the decomposition of $F^0=F_1+F_2+F_3$ given in (\ref{def1}) we only need to check that  
$F_i(\EE,\Omega,f)(-w)=-F_i(-\EE, \Omega, \tilde f)(w), \, \text{ for } i=1,2,3.$ 
Since $\tilde f ^\prime (w)=-f^\prime(-w)$ we have 
\begin{eqnarray*}
F_1(-\EE, \Omega, \tilde f)(w)
&=&2\Omega\big(-\EE\overline w +\EE^2\tilde f(\overline w)-d)w
(1-\EE \tilde f^\prime(w) \big)-\tilde f^\prime(w)\\
&=&-\big[2\Omega\big(\EE(- \overline w) +\EE^2 f(-\overline w)-d)(-w)(1+\EE  f^\prime(-w) \big)-f^\prime(-w)\big]\\
&=&-F_1(\EE, \Omega, f)(-w).
\end{eqnarray*}
Straightforward computations will lead to the same properties for the functions $F_2$ and $F_3.$
It follows that,
$$
F^0(\EE,\Omega,f)(w)=-F^0(-\EE, \Omega, \tilde f)(-w)
$$
and therefore $(-\EE, \tilde f)$ defines a curve of solutions for $0<\EE<\EE_0.$
 \\
 \\
 {\bf{(iv)}} As before we shall restrict the discussion to the case $\alpha=0$ because the proof 
 in the case $\alpha\in(0,1)$  follows  exactly the same lines. First we shall make the following 
 comment about the regularity of the conformal mapping. As it was mentioned in  Remark \ref{Reg-up} one 
 can reproduce the preceding proofs and  replace $\beta$ by  $n+\beta$ with $n\in\NN$. 
 Therefore the implicit function theorem ensures that the function $f_\EE$ belongs to $C^{n+1+\beta}$ for 
 any fixed $n$. Of course, the size of $\EE_0$ is not uniform with respect to $n$ and it shrinks to
 zero as $n$ grows to infinity. Now to prove the convexity of the domain $D_1^\EE$ we shall implement 
 the same arguments  of \cite{HMV}. Recall  that the exterior  conformal mapping associated to this 
 domain is given by
$$
\phi(w)=\EE w+\EE^2 f_\EE(w)
$$
and  the curvature can be expressed by the formula
$$
\kappa(\theta)=\frac{1}{|\phi^\prime(w)|}\operatorname{Re}\Big(1+w\frac{\phi^{\prime\prime}(w)}{\phi^\prime(w)}\Big).
$$
It is plain that 
$$
1+w\frac{\phi^{\prime\prime}(w)}{\phi^\prime(w)}=1+\EE w
\frac{f^{\prime\prime}(w)}{1+\EE f^\prime(w)}
$$
and so
$$
\operatorname{Re}\Big(1+w\frac{\phi^{\prime\prime}(w)}{\phi^\prime(w)}\Big)
\geq 1-|\EE| \frac{|f^{\prime\prime}(w)|}{1-|\EE|f^\prime(w)|} \geq 1-\frac{|\EE|}
{1-|\EE|},
$$
which is non-negative if $|\EE| < 1/2.$ Thus the curvature is strictly positive and therefore the domain is strictly convex.

\end{proof}

  \section{Existence of counter-rotating vortex pairs}\label{count-E}

  In this section we will prove the existence of planar translating pairs of vortex patches 
  with velocity $U$ in the direction $(OY)$  for the $(\hbox{SQG})_\alpha$ with $\alpha\in[0,1)$. 
  The proof is similar to that of  the corotating pairs and therefore we shall skip many details 
  and focus on the significant variations.
  
\subsection{Extension and regularity of $G^\alpha$}
This section is devoted to the study of  the regularity of the functions $G^\alpha$ appearing 
in \eqref{DDX1} and \eqref{split1X}. 
\begin{proposition}\label{propXX}
 Let $\alpha\in [0,1)$, then the following holds true.
\begin{enumerate}
\item The function $G^\alpha$ can be extended from $(-\frac12,\frac12)\times\RR\times B_1^\alpha\to Y^\alpha$ as a $C^1$ function.
\
Moreover,  for any $U\in\RR,\,$ the operator  
$ \partial_fG^\alpha(0,U,0):X^\alpha\to \widehat Y^\alpha$ is an isomorphism. 
More precisely, for  $\displaystyle{h=\sum_{n\geq1} a_n {w^{-n}}\in X^\alpha}$, we get
 $$
 \partial_fG^\alpha(0,U,0)h(w)=-\sum_{n\geq1} a_n \widehat\gamma_n e_{n+1},
 $$
  with
 \begin{eqnarray*}
\widehat \gamma_n&=&\frac{\widehat{C}_\alpha\Gamma(1-\alpha)}{4\Gamma^2(1-\frac\alpha2)}\bigg( \frac{2(1+n)}{1-\frac\alpha2}-\frac{\big(1+\frac{\alpha}{2}\big)_{n}}{\big(1-\frac{\alpha}{2}\big)_{n}}-\frac{\big(1+\frac{\alpha}{2}\big)_{n+1}}{\big(1-\frac{\alpha}{2}\big)_{n+1}}\bigg).
\end{eqnarray*}
\item Two initial point vortex $\pi \delta_0$ and $-\pi \delta_{(2d,0)}$  move  uniformly in the direction $(OY)$ with the speed
 $$
U_{sing}^\alpha\equiv \frac{\widehat{C}_\alpha}{2(2d)^{1+\alpha}}\cdot
 $$

\end{enumerate}

\end{proposition} 
\begin{proof}
${\bf{(i)}}$ The proof is quite similar to $(i)$ of Proposition \ref{prop1}. The only slight difference is in the treatment of $G_1$ which is clearly $C^1$ in the variable  $\EE$ and moreover it has a polynomial  dependence in $\Omega, f, f^\prime,$ and so its derivatives in these variables are also continuous. Note that $G_2$ and $G_3$ are exactly the same functions of   the rotating case, see \eqref{split1}. Now we shall compute the linearized operator and we restrict ourselves only to $\alpha\in (0,1)$. The case $\alpha=0$ can be done separately or by taking the limit when $\alpha\to 0^+$.  It is easy to see that
$$
\partial_f G_1(0,U,0)=0,
$$
and so
\begin{eqnarray*}
\partial_f G^\alpha(0,U,0)&=&\partial_f G_1(0,U,0)+\partial_f G_2(0,0)+\partial_f G_3(0,0)\\
&=&\partial_f G_2(0,0)+\partial_f G_3(0,0).
\end{eqnarray*}
On the other hand by (\ref{F_2}) $\partial_f G_3(0,0)=0, $ and so this operator coincides, after a change of sign,  with the linearized operator in the rotating case and whose formula was stated in Proposition \ref{prop1}.
\\
{\bf{(ii)}} Obvious computations yield 
$$
G_1(0,U,0)=U e_1.
$$
According to  \eqref{G_2}  one finds
$$
G_3(0,0)=-\frac{\alpha C_\alpha}{2(2d)^{1+\alpha}} e_1.
$$
Using \eqref{FF1}, \eqref{F001} and \eqref{I12} we obtain
$$
G_2(0,0)=0.
$$
Therefore we get
 $$
 G^\alpha(0,U,0)=\Big(U-\frac{\alpha C_\alpha}{2(2d)^{1+\alpha}}\Big) e_1
 $$
and consequently $G^\alpha(0,U,0)=0$ if and only if
$$
U=U_{sing}^\alpha\equiv\frac{\widehat{C}_\alpha}{2(2d)^{1+\alpha}}\cdot
$$

\end{proof}
\subsection{Relationship  between the speed  and the boundary shape}\label{real-U}
As for the rotating case the image of the space  $X^\alpha$ by $G^\alpha (\EE,U,\cdot)$ is contained in  $Y^\alpha $ and not necessary in $\widehat{Y}^\alpha.$ Therefore to apply the implicit function theorem we should impose a constraint between $U, \EE$ and $f$ which guarantees a range  contained in $\widehat{Y}^\alpha.$ The main  result reads as follows.
 \begin{proposition}\label{propXX1}
 Let $\alpha\in [0,1)$. There exists a $C^1$ function $\mathcal{R}^\alpha:(-\frac12,\frac12)\times B_1^\alpha\to \RR$  such that with the choice 
 \begin{eqnarray*}
U=U^\alpha(\EE,f)
=U_{sing}^\alpha+\mathcal{R}^\alpha (\EE,f),
\end{eqnarray*}
 the function $\widehat{G}^\alpha:(-\frac12,\frac12)\times B_1^\alpha\to \widehat{Y}^\alpha$ given  by
 $$
\widehat{G}^\alpha(\EE,f)\equiv G^\alpha\big(\EE,U(\EE,f),f\big)
 $$
is well-defined and is $C^1$. Moreover,
$$
\forall f\in B_1^\alpha, \,\mathcal{R}^\alpha (0,f)=0\quad\hbox{and}\quad \widehat{G}^\alpha(0,0)=0.
$$

 \end{proposition}
 \begin{proof}
 We shall follow the same strategy of the subsection \ref{real-velo}. According to  Proposition \ref{propXX} the 
linear operator $\partial_fG^\alpha(0,U,0)$ is an isomorphism from $X^\alpha$ to $\widehat Y^\alpha$, and the latter space  is strictly contained in $Y^\alpha.$ 
Note also  that the image of the  functional $G^\alpha$ lies  in $Y^\alpha$ and therefore  we shall 
 choose $U$ in such a way that the range  of $G^\alpha$ is contained in $\widehat Y^\alpha.$ 
 To this end  we should impose a nonlinear  constraint on $U$ such that the coefficient of $ e_1$ vanishes  in the Fourier expansion of $G^\alpha (U,\EE,f)$. This constraint reads
 \begin{equation}\label{condY}
\fint_{\mathbb{T}}G^\alpha\big(\EE,U,f(w)\big)dw=0.
\end{equation}
\\
{\bf{Case $\alpha=0$}}. According to \eqref{DDX1}, the above condition is equivalent to  
\begin{equation}\label{const1}
\fint_{\mathbb{T}}F^0\big(U,\EE,f(w)\big)\big(\overline{w}^2-1\big) dw=0.
\end{equation}
Note that by residue theorem
$$
\fint_{\mathbb{T}}F_1\big(U,\EE,f(w)\big)\overline{w}^2 dw=2U
$$
and
\begin{eqnarray*}
\fint_{\mathbb{T}}F_1\big(U,\EE,f(w)\big) dw&=&2U\EE\fint_{\mathbb{T}} w\,f^\prime(w) dw\\
&=&-2U\EE \fint_{\mathbb{T}} f(w) dw.
\end{eqnarray*}
Consequently
\begin{eqnarray*}
\fint_{\mathbb{T}}F_1\big(\Omega,\EE,f(w)\big)\big(\overline{w}^2-1\big) dw=2U\Big(1+\EE\fint_{\mathbb{T}} f(w) d w\Big).\end{eqnarray*}
Now we shall calculate the contribution of $F_3.$ First we make  the decomposition
$$
F_3(\EE,f(w))=\widehat{F_3}(w) w\big(1+\EE f^\prime(w)\big),$$
with
$$ \widehat{F_3}(\EE,f(w))\equiv\fint_{\mathbb{T}}\frac{\overline{\tau}+\EE{f(\overline\tau)}}{\EE(\tau+w)+\EE^2 \big(f(\tau)+f(w)\big)-2d}\big(1+\EE f^\prime(\tau)\big)d\tau.
$$
Now one may write
\begin{eqnarray*}
\frac{\overline{\tau}+\EE{f(\overline\tau)}}{\EE(\tau+w)+\EE^2 \big(f(\tau)+f(w)\big)-2d}&=&-\frac{\overline{\tau}}{2d}+\EE\frac{f(\overline{\tau})}{\EE(\tau+w)+\EE^2 \big(f(\tau)+f(w)\big)-2d}\\
&+&\frac{\EE}{2d}\frac{\tau+w+\EE(f(\tau)+f(w))}{\EE(\tau+w)+\EE^2 \big(f(\tau)+f(w)\big)-2d}\overline{\tau}\\
&\equiv&-\frac{\overline{\tau}}{2d}+\EE g_3(\EE,\tau,w).
 \end{eqnarray*}
 Thus we find
 $$
  \widehat{F_3}(\EE,f(w))=-\frac{1}{2d}+\EE\int_{\mathbb{T}}g_3(\EE,\tau,w)\big(1+\EE f^\prime(\tau)\big) d\tau.
 $$
 Hence
\begin{eqnarray*}
 \fint_{\mathbb{T}}F_3\big(\Omega,\EE,f(w)\big)\overline{w}^2 dw&=&-\frac{1}{2d}+\EE\fint_{\mathbb{T}}\fint_{\mathbb{T}}g_3(\EE,\tau,w)\big(1+\EE f^\prime(\tau)\big)\overline{w}\big(1+\EE f^\prime(w)\big) d\tau dw\\
 \fint_{\mathbb{T}}F_3\big(\Omega,\EE,f(w)\big) dw&=&\frac{\EE}{2d}\fint_{\mathbb{T}} f(\tau) d\tau+\EE\fint_{\mathbb{T}}\fint_{\mathbb{T}}g_3(\EE,\tau,w)\big(1+\EE f^\prime(\tau)\big){w}\big(1+\EE f^\prime(w)\big) d\tau dw.
\end{eqnarray*}
It follows that
\begin{eqnarray*}
 \fint_{\mathbb{T}}F_3\big(\Omega,\EE,f(w)\big)\big(\overline{w}^2-1\big) dw&=&-\frac{1}{2d}-\frac{\EE}{2d}\fint_{\mathbb{T}} f(\tau) d\tau\\
 &+&\EE\fint_{\mathbb{T}}\fint_{\mathbb{T}}g_3(\EE,\tau,w)\big(1+\EE f^\prime(\tau)\big)(\overline{w}-w)\big(1+\EE f^\prime(w)\big) d\tau dw.
 \end{eqnarray*}
 On the other hand using residue  theorem we get
 \begin{eqnarray*}
 F_2(\EE,f(w))&=&\EE \fint_{\mathbb{T}}\frac{A\overline{B}-\overline{A} B}{A(A+\EE B)} f^\prime(\tau)d\tau\, w\big(1+\EE f^\prime(w)\big)\\
 &+&\EE\fint_{\mathbb{T}}\frac{\big(\overline{A}B-A\overline{B}\big) B }{A^2(A+\EE B)}d\tau\, w\big(1+\EE f^\prime(w)\big)\\
 &\equiv&\EE g_2(\EE,w)w\big(1+\EE f^\prime(w)\big).
  \end{eqnarray*}
 Thus
 \begin{equation*}
 \fint_{\mathbb{T}}F_2\big(\Omega,\EE,f(w)\big)\overline{w}^2 dw- \fint_{\mathbb{T}}F_2\big(\Omega,\EE,f(w)\big) dw=\EE\fint_{\mathbb{T}}g_2(\EE,w)(\overline{w}-w)\big(1+\EE f^\prime(w)\big)  dw.
 \end{equation*}
The equation \eqref{const1} becomes
\begin{eqnarray*}
 2U\Big(1+\EE\fint_{\mathbb{T}}f(w) dw\Big)&=&\frac{1}{2d}+\frac{\EE}{2d}\fint_{\mathbb{T}} f(\tau) d\tau\\&+&\EE\fint_{\mathbb{T}}g_2(\EE,w)(w-\overline{w})\big(1+\EE f^\prime(w)\big)  dw\\
 &+&\EE\fint_{\mathbb{T}}\fint_{\mathbb{T}}g_3(\EE,\tau,w)\big(1+\EE f^\prime(\tau)\big)(w-\overline{w})\big(1+\EE f^\prime(w)\big) d\tau dw\\
 &\equiv&\frac{1}{2d}+\frac{\EE}{2d} T_1(\EE,f)
  \end{eqnarray*}
  which may be written in the form
 \begin{eqnarray}\label{const2}
 \nonumber U&=&U^0(\EE,f)\\
 \nonumber&=&\frac{1}{4d}\frac{1+\EE T_1(\EE,f)}{1+\EE T_2(f)}\\
  \nonumber&=&U_{sing}^0+\frac{\EE}{4d}\frac{T_1(\EE,f)-T_2(f)}{1+\EE T_2(f)}\\
  &\equiv&U_{sing}^0+\mathcal{R}^0(\EE,f),
  \end{eqnarray}
  with
  $$
  T_2(f)=\fint_{\mathbb{T}}f(w) dw.
  $$
  \\
  {\bf{Case $\alpha\in (0,1)$.}}  
 From the splitting  \eqref{split1X} the assumption \eqref{condY} is equivalent to
\begin{equation}\label{restr1X}
A_0=-A_1-A_2,
\end{equation}
with 
$$
A_j=-2i\fint_{\mathbb{T}}G_j\big(\EE,\Omega,f(w)\big)) dw.
$$
Note that $A_j$ is the Fourier coefficient of $e_1=\hbox{Im}(w)$ in $G_j\equiv \hbox{Im}(F_j),$ then
$$
A_j=\fint_{\mathbb{T}}F_j\big(\EE,\Omega,f(w)\big))\big(\overline{w}^2-1\big) dw.
$$
The computation of $A_1$ is easy,
$$
A_1=-U\fint_{\mathbb{T}}\big(1+\EE|\EE|^\alpha f^\prime(\overline{w})\big)\big(\overline{w}^3-\overline{w}\big)dw.
$$
Since 
$$
f(w)=\sum_{n\geq1} a_n \overline{w}^n\quad\hbox{and}\quad f^\prime(w)=-\sum_{n\geq1} na_n  \overline{w}^{n+1},
$$
then 
\begin{eqnarray*}
\fint_{\mathbb{T}}f^\prime(\overline{w})\big(\overline{w}^3-\overline{w}\big)dw&=&-a_1\\&=&-\fint_{\mathbb{T}} f(\tau) d\tau.
\end{eqnarray*}
Consequently
\begin{equation}\label{A0X}
A_1=U\Big(1+\EE|\EE|^\alpha\fint_{\mathbb{T}}f(\tau) d\tau\Big).
\end{equation}
Combining \eqref{restr1X} with \eqref{A0X}, \eqref{A_1}  and \eqref{A_2} one gets
$$
U\Big(1+\EE|\EE|^\alpha\fint_{\mathbb{T}}f(\tau) d\tau\Big)=\frac{\alpha C_\alpha}{2(2d)^{1+\alpha}}+\EE \mathcal{R}_2(\EE,f)-\EE|\EE|^\alpha\mathcal{R}_1(\EE,f).
$$
Thus
\begin{eqnarray*}
U&=&U^\alpha(\EE,f)\\
&\equiv&\frac{\frac{\alpha C_\alpha}{2(2d)^{1+\alpha}}+\EE \mathcal{R}_2(\EE,f)-\EE|\EE|^\alpha\mathcal{R}_1(\EE,f)}{1+\EE|\EE|^\alpha\fint_{\mathbb{T}}f(\tau) d\tau}\\
&\equiv&\frac{\alpha C_\alpha}{2(2d)^{1+\alpha}}+\mathcal{R}^\alpha(\EE,f).
\end{eqnarray*}
Similarly to the rotating case, $\mathcal{R}^\alpha:(-\frac12,\frac12)\times B_1^\alpha\to \RR$ is well-defined and $C^1.$ 
Moreover, by construction  one can see that $\mathcal R^\alpha(0,f)=0. $ This will imply that 
  $\widehat G^\alpha(0,0)=0.$

\end{proof}
\subsection{Proof of the Main Theorem-$\hbox{(ii)}$}
Recall that the V-states equation can be written in the form
$$
\widehat{G}^\alpha(\EE,f)=0, \quad (\EE,f)\in (-1/2,1/2)\times B_1^\alpha,
$$
with $\widehat{G}^\alpha$ being the functional defined by
$$
\widehat{G}^\alpha\big(\EE,f(w)\big)=G^\alpha\big(\EE,U^\alpha (\EE,f),f\big).
$$
The proof of the existence of translating pairs stated in the Main Theorem follows from the next
proposition  whose  proof is quite similar to that of Proposition \ref{propY1} and so is
left to the reader.
\begin{proposition}
Let $\alpha\in [0,1)$. The following holds true.
\begin{enumerate}
\item The linear operator $\partial_f \widehat G^\alpha(0,0):X^\alpha\to \widehat{Y}^\alpha$ is an isomorphism and
$$
 \partial_f\widehat{G}^\alpha(0,0)h(w)=-\sum_{n\geq1} a_n \widehat\gamma_n e_{n+1} 
 $$
  with
$$
\widehat \gamma_n=\frac{\widehat{C}_\alpha\Gamma(1-\alpha)}{4\Gamma^2(1-\frac\alpha2)}\bigg( \frac{2(1+n)}{1-\frac\alpha2}-\frac{\big(1+\frac{\alpha}{2}\big)_{n}}{\big(1-\frac{\alpha}{2}\big)_{n}}-\frac{\big(1+\frac{\alpha}{2}\big)_{n+1}}{\big(1-\frac{\alpha}{2}\big)_{n+1}}\bigg).
$$
\item There exists $\EE_0>0$ such that the set 
$$
\Big\{(\EE,f)\in [-\EE_0,\EE_0]\times B_1^\alpha,\, s.t.\quad \widehat{G}^\alpha(\EE,f)=0\Big\}
$$
is parametrized by one-dimensional curve $\EE \in[-\EE_0,\EE_0]\mapsto (\EE, f_\EE)$ and
$$
\forall \EE\in [-\EE_0,\EE_0]\backslash\{0\},\, f_\EE\neq0.
$$

\item If $(\EE,f)$ is a solution then $(-\EE,\tilde{f})$ is also a solution,  where
$$
 \tilde{f}(w)=f(-w), \quad \forall w\in \mathbb{T}.
$$
\item For all $\EE \in[-\EE_0,\EE_0],$ the domain $D_1^\EE$ is strictly convex.

\end{enumerate}

\end{proposition}
 \begin{ackname}
This work was partially supported by the grants of Generalitat de Catalunya 2014SGR7,
Ministerio de Economia y Competitividad MTM 2013-4469,
ECPP7- Marie Curie project MAnET  and the ANR project Dyficolti ANR-13-BS01-0003- 01.\end{ackname}

 \end{document}